\documentclass[10pt,reqno]{amsart}
\usepackage{palatino}
\usepackage{amsmath,amssymb,mathrsfs}

\newtheorem{thm}{Theorem}
\newtheorem{lemma}[thm]{Lemma}

\newtheorem{prop}[thm]{Proposition}
\newtheorem*{thm A}{ Theorem A}
\newtheorem*{thm B}{ Theorem B}

\newcommand{\N}{{\mathbb N}}
\newcommand{\Z}{{\mathbb Z}}
\newcommand{\C}{{\mathbb C}}
\newcommand{\s}{{\mathbb S}}
\newcommand{\R}{{\mathbb R}}

\newcommand{\T}{{\mathbb T}}
\newcommand{\D}{{\mathbb D}}

\newcommand{\be}{\mathcal{B}}

\newcommand{\HH}{\mathcal H}

\newcommand{\CC}{\mathcal C}
\newcommand{\RR}{\mathcal R}
\newcommand{\Sl}{\mathcal S}
\newcommand{\mm}{d\mu_{m,\alpha}(\zeta)}

\newcommand{\BB}{\mathcal B}

\begin{document}

\title[Hankel operators on vector-valued generalized Fock spaces on $\C^d$]{Small Hankel operators on vector-valued generalized Fock spaces on $\C^d$}

\author{H. Bommier-Hato} 

%\address{Bommier-Hato: Aix-Marseille Universit\'e, I2M UMR CNRS 7373, 39~Rue F.~Joliot-Curie,
%13453~Marseille Cedex~13, France}
%\email{ helene.bommier@gmail.com}
\address{Bommier-Hato: Faculty of Mathematics,
University of Vienna,
Oskar-Morgenstern-Platz 1, 1090 Vienna, Austria } \email{helene.bommier@gmail.com}

\thanks{The author was supported by the FWF project P 30251-N35.}



\maketitle

%%%%%%%%%%%%%%%%%%begining %%%%%%%%%%%%%%%%%%%%%%%%%%%%%%%%%%%%%
 
\begin{abstract} 
We study  small Hankel operators $h_b$ with operator-valued holomorphic symbol $b$ on a class of vector-valued Fock type spaces. We show that the boundedness / compactness of $h_b$ is equivalent to the membership of $b$ to a specific growth space, which is described via  a Littlewood-Paley type condition and  a Bergman type  projection, and estimate the norm of $h_b$. We also establish some properties of duality and density for these Fock spaces.

\end{abstract}

\section{Introduction}
Due to its numerous applications, the theory of  Hankel operators has been developped in various directions for the last decades. For instance, Hankel operators on the Hardy space are related to control theory, and the vectorial setting is used in non-commutative analysis, such as the theory of approximation by analytic 
matrices (see \cite{Nik,Pell2}). Small Hankel operators with analytic symbols on vector-valued  spaces have been studied  on Bergman spaces \cite{AlCo,Co}, and  on weighted Dirichlet spaces of the unit disk of $\C$ \cite{AlPe}.  Though the boundedness and the compactness  are characterized in a similar way  for scalar and  vectorial functions, namely the membership of the symbol to some Bloch space,  the vectorial case brings new difficulties.\\
In the context of entire functions, the most classical framework is the Segal-Bargmann space $F^{2}_{\alpha}=F^{2}_{\alpha}(\C^d,\C)$ of Quantum Mechanics (see \cite{JPR,Zf}), which is the space of all entire functions $f:\C^d\rightarrow\C$, $d\geq 1$, such that $\left|f\right|^2$ is integrable with respect to the
Gaussian 
$$d\mu_{\alpha}(z):=\left(\frac{\alpha}{\pi}\right)^d e^{-\alpha\left|z\right|^2}dv(z),$$
where $dv(z)$ stands for the Lebesgue volume on $\C^d$ and $\alpha>0$ \cite{Foll,JPR,Zf}. 
 
 Given real numbers $m\geq 1$, $\alpha>0$,  we consider the probability measure on $\C^d$
\begin{equation}\label{dmu cm alpha}
d\mu_{m,\alpha}(\zeta):=c_{m,\alpha}e^{-{\alpha}\left|\zeta\right|^{2m}}dv(\zeta),\ \text{where  }
	\ c_{m,\alpha}=\frac{m\alpha^{d/m}}{\pi^d}\frac{\Gamma(d)}{\Gamma\left(\frac{d}{m}\right)}. 
\end{equation}
%\begin{align}\label{dmu cm alpha1}
	%d\mu_{m,\alpha(\zeta)}&:=c_{m,\alpha}e^{-{\alpha}\left|\zeta\right|^{2m}}dv(\zeta), \ m\geq 1,\  \alpha>0,\\
	%\text{with  }c_{m,\alpha}&=\frac{m\alpha^{d/m}}{\pi^d}\frac{\Gamma(d)}{\Gamma\left(\frac{d}{m}\right)}. 
%\end{align}
For a complex Banach space $Y$, and $1\leq p \leq\infty$, $L^{p}_{m,\alpha}( Y):=L^{p}_{m,\alpha}(\C^d, Y)$ is the space of  $Y$-valued strongly measurable functions $f$ on $\C^d $ such that
\begin{gather*}
%\label{norm Fp Y}
	\left\|f\right\|_{L^{p}_{m,\alpha}( Y)}^p:=c_{m,\alpha}\int_{\C^d}\left\|f(\zeta) e^{-\frac{\alpha}{2}\left|\zeta\right|^{2m}}\right\|_Y^pdv(\zeta)<\infty,\ 1\leq p<\infty,\\
	\left\|f\right\|_{L^{\infty}_{m,\alpha}( Y)}:=\text{ess} \sup_{\zeta\in \C^d}\left\|f(\zeta) e^{-\frac{\alpha}{2}\left|\zeta\right|^{2m}}\right\|_Y,
\end{gather*}
and  $F^{p}_{m,\alpha}( Y):=F^{p}_{m,\alpha}(\C^d, Y)$ is the space of  entire functions which are in $L^{p}_{m,\alpha}( Y)$, equipped with the same norm. When $Y =\C, $ we
shall simply denote these spaces by $L^{p}_{m,\alpha}$ and $F^{p}_{m,\alpha}$.\\
Moreover, $F^{\infty,0}_{m,\alpha}( Y)$
is defined as  the space of entire functions $f$ such that
\begin{equation*}
%\label{F infty 0}
	\lim_{|\zeta|\rightarrow +\infty}\left\|f(\zeta)\right\|_Y e^{-\frac{\alpha}{2}\left|\zeta\right|^{2m}}=0.
\end{equation*}
 
In this paper,  $\HH$ is a separable Hilbert space. We are interested in Hankel operators on the vector-valued Fock type spaces   $F^{p}_{m,\alpha}(\C^d, X)$, when $X$ is the  space $\HH$ or a Schatten-class ideal, and the symbol is an entire operator-valued function. In the scalar case, Hankel operators on $F^{2}_{\alpha}$ have been considered in \cite{JPR,Zf}
%, and on $F^{p}_{2l,\alpha}(\C, \C)$ for  an integer $l$  in the preprint \cite{Cas}
 (our definition of $h_b$ may differ from theirs up to unitary operators). \\
To our knowledge, the vectorial setting has not yet been considered.
%it is likely that their methods could be adapted to the scalar  space $F^{p}_{m,\alpha}(\C, \C)$.
 It is worthwhile mentioning that the methods used in the scalar case  do not apply when the target space $X $ has infinite or finite dimension $\geq 2$. Besides, for general $m> 1$, the framework $F^{p}_{m,\alpha}(\C^d, X)$,  $d\geq 2$ is more involved than the case of one variable, and $m=1$. Moreover, we  also study the growth of the  operator norm of $h_b$  with respect to $m$  (see details below).\\

%we also show that the operator norm of $h_b$ grows as a power function of $m$

% $\frac{1}{p}+\frac{1}{p'}=1$, with the convention $\frac{1}{\infty}=0$)\\
Pointwise estimates (see Proposition \ref{pointwise Fp}) imply that $F^{p}_{m,\alpha}( Y)$ is a Banach space. For $1\leq p \leq\infty$,  $p'$ is the conjugate exponent of $p$, satisfying $\frac{1}{p}+\frac{1}{p'}=1$.
 The duality  $\left\langle .,.\right\rangle_{Y-Y^*}$, between    $Y$ and its  dual space  $Y^*$,  gives rise to the natural duality (see Proposition \ref{dual Fp 1 leq p})
\begin{equation*}
%\label{dual alpha}
 \left\langle x,y\right\rangle_{\alpha}:= \int_{\C^d}\left\langle x(\zeta),y(\overline \zeta)\right\rangle_{Y-Y^*}d\mu_{m,\alpha}(\zeta),\ x\in F^{p}_{m,\alpha}( Y),\ y\in F^{p'}_{m,\alpha}( Y^*).
\end{equation*} 

 %$$\left|\left\langle x,y\right\rangle_{\alpha}\right|\leq
%\left\|x\right\|_{F^{p}_{m,\alpha}(\C^d,Y )}\left\|y\right\|_{F^{p'}_{m,\alpha}(\C^d, Y^*)}$$
 %whenever $x\in F^{p}_{m,\alpha}(\C^d, Y)$ and $y\in F^{p'}_{m,\alpha}(\C^d, )$, .
In particular, $F^{2}_{m,\alpha}( \HH)$ is a Hilbert space, with inner product 
%$\left\langle .,.\right\rangle_{\alpha}$,
\begin{equation*}
%\label{inner F2}
 \left\langle x,y\right\rangle_{\alpha}:= \int_{\C^d}\left\langle x(\zeta),y(\zeta)\right\rangle_{\HH}d\mu_{m,\alpha}(\zeta),\ x,y\in F^{2}_{m,\alpha}(\HH),
\end{equation*}  
%\begin{equation}\label{inner F2}
and the orthogonal projection $P_{\alpha}:L^{2}_{m,\alpha}(\HH)\rightarrow  F^{2}_{m,\alpha}( \HH)$ is defined by
\begin{equation*}
%\label{P alpha}
	P_{\alpha}x(z)=\int_{\C^d}x(\zeta)K_{m,\alpha}(z,\zeta)d\mu_{m,\alpha}(\zeta),\ x\in L^{2}_{m,\alpha}( \HH),\  z\in\C^d,
\end{equation*}
 where  $K_{m,\alpha}$ denotes the  reproducing kernel in $F^{2}_{m,\alpha}$.\\

For a Banach space $X$, let $\be(X)$ be the set of all bounded operators on $X$. 

If $b$ is in $F^{\infty}_{m,\alpha}(\be(\HH) )$, the \textit{ small Hankel operator} $h_b$ is defined via the Hankel form
\begin{equation*}
%\label{Hankel form}
	\left\langle h_b x,y \right\rangle_{\alpha}=\int_{\C^d}\left\langle b(\zeta)x(\overline\zeta), y(\zeta)\right\rangle_{\HH}d\mu_{m,\alpha}(\zeta),
\end{equation*}
or via the projection $P_{\alpha}$
\begin{equation}\label{Hankel op}
h_b x=P_{\alpha}\left(bJx\right),
\end{equation}
where $x,y$ are $\HH$-valued polynomials and $(Jx)(z)=x(\overline z)$, $z\in\C^d$. \\
 
Our aim is to  characterize the boundedness and compactness of  small vectorial Hankel operators   on $F^{p}_{m,\alpha}( X)$, where $X$ is the Hilbert space $\HH$, or the Schatten class $\Sl^p(\HH)$.  For  a general Banach space  $Y$,  we   describe some density and duality results for the spaces $F^{p}_{m,\alpha}( Y)$, $1\leq p \leq\infty$ (section \ref{prelim}). Apart from being of independent interest, those properties enable us to give necessary and sufficient conditions for a  Hankel operator $h_b$ to be bounded (section \ref{boundedness}), or compact (section \ref{compact}) on $F^{p}_{m,\alpha}( X)$, when the symbol $b$ is a $\be(\HH)$-valued entire function.  In addition, we use precise estimates for kernel functions to show that 
%(? when the parameter $\alpha$ is bounded away from $0$),
 the norm of $h_b$ is equivalent to the norm $\left\|b\right\|_{F^{\infty}_{m,{\alpha}/{2}}( \be(\HH))}$ (section \ref{boundedness}). Moreover, we  give a characterization of the growth spaces  $ F^{\infty}_{m,\frac{\alpha}{2}}( Y) $ and $ F^{\infty,0}_{m,\frac{\alpha}{2}}( Y) $ involving a Bergman-type projection, and  a Littlewood-Paley type condition (Propositions \ref{charact F infty} and \ref{charact F infty 0}).\\

 We also consider the Hankel operator of symbol $T(b):\C^d\rightarrow\BB(\BB(\mathcal{H}))$  defined by
\begin{equation*}
\label{def tb1}
T(b)(z)S=b(z)S,\ S \in\BB(\mathcal{H}). 
\end{equation*}
 
For $1\leq p<\infty$, $\Sl^p=\Sl^p(\HH)$ denotes  the Schatten-von Neumann class, consisting in all operators $T$ on $\HH$  such that the sequence of the singular values of $\left|T\right|$ are in $l^p$.   $\Sl^{\infty}=\BB(\HH)$), and $\mathcal{K}(\HH)$ is the set of all compact operators on $\HH$. For the duality $\Sl^p-\Sl^{p'}$, we will use the notation
\begin{equation*}
%\label{trace}
	\left\langle A,B\right\rangle_{\text{tr}}=\text{tr}\left(AB^{*}\right),\ A\in\Sl^p,\ B\in\Sl^{p'}. 
\end{equation*}
%For $1\leq p<\infty$, 
On another hand, $G^p$ stands for one of the spaces 
\begin{equation}\label{Gp}
	G^p:=F^{p}_{m,\alpha}(\mathcal{S}^q(\mathcal{H})),\text{ where }q\in\left\{p,2\right\}.
\end{equation}
%We set 

Throughout the paper, $X,Y$ are complex Banach spaces, and $\alpha,\beta$  are positive real numbers. The norm on $\C^d$ is denoted simply by $\left|.\right|$, and the standard inner product by $\left\langle .,.\right\rangle$. The letter $C$ will stand for a positive constant, which may change from line to line, and whose dependence on parameters will be made precise if needed. For two functions $f, g$, the notation $f=O(g)$ or $f \lesssim g$,  means that there exists a constant $C$ such that $f\leq C g$ .
 If  $f=O(g)$ and $g=O(f)$, we  write $f\asymp g$.\\

  The following theorem characterizes the boundedness of the Hankel operator $h_b$.
%	We are also interested in the dependence of $\left\|h_b\right\|$ and $\left\|h_{T(b)}\right\|$ on the parameter $m$,
	Assuming
\begin{equation*}
%\label{alpha borne}
	m\geq 1\text{ and }0<\alpha_0 \leq\alpha\leq\alpha_1,
\end{equation*}
for some positive constants $\alpha_0 $, $\alpha_1$, we also show that the ratio $\left\|h_b\right\|^{-1}  \left\|b\right\|_{ F^{\infty}_{m,\frac{\alpha}{2}}(\C^d, \be(\HH))}$ grows as a power function of $m$.

%Our main result is stated as follows.
\begin{thm A}\label{thm A}
Suppose $1\leq p$, and $b\in F^{\infty}_{m,\alpha}(\C^d, \be(\HH))$. The following statements are equivalent:
\begin{enumerate}
	\item [(a)] $b\in F^{\infty}_{m,\frac{\alpha}{2}}(\C^d, \be(\HH))$;
	\item [(b)] $h_b$ is bounded  on $F^{p}_{m,\alpha}(\C^d, \HH)$;
	\item [(c)] $h_{T(b)}$ is bounded  on $G^p$.
	%$F^{p}_{m,\alpha}(\C^d, \Sl^p(\HH))$.
	%\item $h_{T(b)}$ is bounded  on $F^{p}_{m,\alpha}(\C^d, \Sl^2(\HH))$.
\end{enumerate}
%If $\left\|h_{T(b)}\right\|$ stands for the norm of $h_{T(b)}$ on on $F^{p}_{m,\alpha}(\C^d, \Sl^p(\HH))$ (resp. on $F^{p}_{m,\alpha}(\C^d, \Sl^2(\HH))$), 
Moreover, the quantities $\left\|b\right\|_{ F^{\infty}_{m,\frac{\alpha}{2}}(\C^d, \be(\HH))}$,  $\left\|h_b\right\|$,  $\left\|h_{T(b)}\right\|_{\be(G^p)}$ are comparable; more precisely
$$ C m^{-d}\left\|b\right\|_{ F^{\infty}_{m,\frac{\alpha}{2}}(\C^d, \be(\HH))}\leq \left\|h_b\right\| \leq \left\|h_{T(b)}\right\|_{\be(G^p)} \leq 2^{d}\left\|b\right\|_{ F^{\infty}_{m,\frac{\alpha}{2}}(\C^d, \be(\HH))},$$
where the  constant $C$  is independent of $m$.
% and
%\begin{equation}\label{sigma}
%	\sigma_p:=\max\left(1,\frac{d+1}{p}\right).
%\end{equation}
\end{thm A}
A "little oh" version of condition (a) from Theorem A characterizes the compactness of $h_b$ (see Theorem B, section \ref{compact}).

The proofs of Theorems A and B  also provide a characterization of functions in   $ F^{\infty}_{m,\frac{\alpha}{2}}( Y) $ and $ F^{\infty,0}_{m,\frac{\alpha}{2}}( Y) $, for any Banach space $Y$. Let us denote by $\mathcal{C}_{b}( Y)$ 
%(resp. $\mathcal{C}_{0}( Y)$) 
the space of all   $Y$-valued bounded functions on $\C^d$. 
%(resp. tend to $0$ as $\left|z\right|\rightarrow\infty$). 
The radial derivative of a holomorphic function  $f$ on $\C^d$ is
\begin{equation*}\label{}
	\RR f(z)=\sum^{d}_{j=1}z_j \frac{\partial f}{\partial z_j}(z),\ z\in\C^d.
\end{equation*}
\begin{prop}\label{charact F infty}
Let $\beta>0$, and suppose that $b$ is an $Y$-valued entire function. Then the following conditions are equivalent:
\begin{enumerate}
	\item [(a)] $b$ is in $F^{\infty}_{m,\beta}(Y)$; 
	 \item [(b)] There exists $c\in \mathcal{C}_{b}( Y)$ such that $b=P_{2\beta}c$;
	\item  [(c)] $\left(1+\left|\zeta\right|^{2m}\right)^{-k}\RR^kb(\zeta)$ is in  $L^{\infty}_{m,\beta}( Y)$ for every (some) nonnegative integer $k$.
\end{enumerate}
 In addition, we  have the equivalence of  norms 
\begin{equation*}
%\label{FSob infty}
	\left\|b\right\|_{F^{\infty}_{m,\beta}( Y)}\asymp\sup_{\zeta\in\C^d}\left(1+\left|\zeta\right|^{2m}\right)^{-k} \left\| \RR^kf(\zeta)\right\|_Ye^{-\frac{\beta}{2}\left|\zeta\right|^{2m}}+\sum^{k-1}_{l=0}\left\|\RR^{l}f(0)\right\|_Y.
\end{equation*}
%$\left\|b\right\|_{F^{\infty}_{m,\beta}(\C^d,Y)}$ and  $\left\|b\right\|_{\FF^{k,\infty}_{m,\beta}(\C^d, Y)}$ are all equivalent.
\end{prop}
%In section we focus on the case of the Segal-Bargmann space , and use a weak factorization to interpret our results in terms of duality . 
Similarly, Proposition \ref{charact F infty 0} characterizes the membership in $F^{\infty,0}_{m,\beta}(Y)$ with a  "little oh" version of condition (b) and (c) from Proposition \ref{charact F infty}.

Thus $F^{\infty}_{m,\beta}(\be(\HH))$ (resp. $F^{\infty,0}_{m,\beta}(\mathcal{K}(\HH))$) can be viewed as a space of symbols of  bounded (resp. compact) Hankel operators, being the analog of  the Bloch space (resp. little Bloch space) of the disc $\D$ (see \cite{Zdisc}), or the unit ball in $\C^d$ (see  \cite{Zhankel} ).
 \\
The equivalence of norms in Proposition \ref{charact F infty} is a  Littlewood-Paley type condition (see similar properties  in \cite{CP}  in the $L^p$ setting, $0< p<\infty$,  and for so-called Fock-Sobolev spaces in \cite{Zcho}).\\
 %We mention that  the scalar space $F^{\alpha}_{2}$ ($m=1$) is presented in  \cite{JPR,Zf}, as well as  little Hankel operators .\\

 Up to our knowledge, our results are new, and we point out that our methods have to take into  account  the vectorial setting, for several variables and a general real number $m\geq1$.\\
First,  the techniques used to study scalar Hankel operators ($X=\C$) rely on the identity
$$ \left\langle P_{\alpha}(b \overline x), y\right\rangle_{\alpha}=\left\langle b,xy \right\rangle_{\alpha},\ x,y\text{ holomorphic},$$
which do not hold in the non-commutative case, when the target space $X$ has  higher dimension.\\
When $m=1$, the reproducing kernels are exponential functions, which gives  exact formulas (see  \cite{Zf}).  When $m>1$, we have to introduce new techniques, based on the    asymptotics of kernel functions.\\
Besides, the proofs in \cite{AlCo,Co}   for the unit disc   involve the derivative $b'$;  to deal with entire functions of  several variables, we use  the radial derivative of $b$. In addition,   we consider the spaces $F^{p}_{m,\beta}( Y)$ for  general exponents $p$ and a Banach space $Y$. Thus, our first step is to establish  duality and density results,   which  provide us with  tools to study the boundedness and the compactness of $h_b$.\\

\section{Preliminaries} \label{prelim}

%$d\sigma(\zeta)$ Lebesgue measure on $\s^{2d-1}$ the unit sphere in $\C^d$ (without normalisation)\\
%$$ \int_{\s^{2d-1}}d\sigma(\zeta)=\frac{2\pi^d}{\Gamma(d)} $$

Since the weight  $e^{-\alpha\left|\zeta\right|^{2m}}$ depends only on $\left|\zeta\right|$, 
the monomials $\left(v_{\nu}\right)_{\nu\in\N^d}$,with $v_{\nu}(\zeta)=\zeta^{\nu}$, form an orthogonal basis in $ F^{2}_{m,\alpha} $.   Here, we use  the standard multiindex notations $\nu!=\nu_1!\cdots\nu_d!$, $\zeta^{\nu}=\zeta_1^{\nu_1}\cdots\zeta_d^{\nu_d}$. It was proved in \cite{BEY,Bo} (the formulas are given with different normalizations) that
\begin{equation*}
%\label{s alpha}
	 s_{\alpha,\nu}:=\left\|\zeta^{\nu}\right\|^{2}_{F^{2}_{m,\alpha}}= C^{-1}_{m}  \frac{\nu!\Gamma\left(\frac{d+\left|\nu\right|}{m}\right)}{\Gamma\left(d+\left|\nu\right|\right)\alpha^{\frac{\left|\nu\right|}{m}}}, \text{ where   }  C_m:=\frac{\Gamma\left(\frac{d}{m}\right)}{\Gamma(d)}.
\end{equation*}
%where 
%\begin{equation}\label{Cm}
%	C_m:=\frac{\Gamma\left(\frac{d}{m}\right)}{\Gamma(d)}.
%\end{equation}
In \cite{BY}, we use  the  theory from Aronszajn \cite{Ar} to compute the reproducing kernel of $ F^{2}_{m,\alpha} $, 
\begin{equation}\label{kernel Fm}
	K_{m,\alpha}(\xi,\zeta)=C_m E^{(d-1)}_{\frac{1}{m},\frac{1}{m}}\left(\alpha^{1/m}\left\langle \xi,\zeta\right\rangle\right),\quad \xi,\zeta\in\C^d,
\end{equation}
where 
\begin{equation*}
%\label{Mittag L}
E_{\beta,\gamma}(z)=\sum^{+\infty}_{k=0}\frac{z^k}{\Gamma\left(\beta k+\gamma\right)},\  \beta,\gamma>0,
\end{equation*}
is the Mittag-Leffler function. Recall that
the  orthogonal projection $P_{m,\alpha}: L^{2}_{m,\alpha}\rightarrow F^{2}_{m,\alpha}$ 
 is given by
%the integral operator with kernel $K_{m,\alpha}$,
\begin{equation*}
%\label{proj p alpha}
P_{m,\alpha}f(\zeta)=\int_{\C^d}K_{m,\alpha}(\zeta,\xi)f(\xi)d\mu_{m,\alpha}(\xi),\quad \zeta\in\C^d.
\end{equation*}
Let $b\in F^{\infty}_{m,\alpha}(\BB(\HH))$. From its definition (\ref{Hankel op}), the Hankel operator $h_b$ is defined by

\begin{equation*}\label{hankel integral}
	h_b x(z)=\int_{\C^d}b(\zeta)x(\overline \zeta)K_{m,\alpha}(z,\zeta)d\mu_{m,\alpha}(\zeta),\quad z\in\C^d,
	\end{equation*}
provided that $\left\|b(.)x(.)K_{m,\alpha}(z,.)\right\|_{\HH}$ is in $L^1(d\mu_{m,\alpha})$ for all $z\in\C^d$. Then, relation (\ref{kernel Fm}) induces that 
  we shall  use
  asymptotics of  Mittag-Leffler function and its derivatives. We start with
%as $\left|z\right|\rightarrow+\infty$
\begin{equation}\label{asymp ML m>1/2}
	E_{\frac{1}{m},\frac{1}{m}}(z)=\begin{cases}
	mz^{m-1}e^{z^m}+O\left(z^{-1}\right),\ \left|\arg z\right|\leq\frac{\pi}{2m}\\
	O\left(\frac{1}{z}\right),\ \frac{\pi}{2m}<|\arg z|<\pi
	\end{cases}
\end{equation}
 %and by
%\begin{equation}\label{asymp ML m leq 1/2}
	%E_{\frac{1}{m},\frac{1}{m}}(z)=m\sum^{N}_{j=-N}z^{m-1}e^{2\pi ij(m-1)}e^{z^m e^{2\pi ijm}}+O\left(z^{-1}\right),\ \pi<\arg z\leq\pi,
	%\end{equation}
%for $0 < m \leq\frac{1}{2}$, 
% where $N$ is the integer satisfying $N <\frac{1}{2m}\leq N+1$,
 %and the powers $z^{m-1}$,  $z^{m}$are the principal branches
for $m > \frac{1}{2}$ and  $z\in\C\setminus\left\{0\right\}$  (see e.g. Bateman and Erdelyi \cite{Bo}, vol. III, §18.1, formulas (21)–(22);  Wong and Zhao \cite{WoZh},
and §5.1.4 in the book by Paris and Kaminski  \cite{PaKa}). 

The expansion can be differentiated termwise any number of times (see \cite{BaEr} for details). An
 induction argument shows that
\begin{equation*}\label{def pk}
	\frac{d^{k-1}}{dz^{k-1}}(mz^{m-1}e^{z^m})=\frac{p_k(z^m)}{z^k}e^{z^m},
\end{equation*}
where $p_k$ is  defined recursively by
\begin{equation*}\label{def p k+1}
p_0 = 1,\  p_{k+1}(x) = (mx-k)p_k(x) +mxp_{k}'(x).
\end{equation*}
Thus, for $k\geq 1$, $p_k$ is a polynomial of degree $k$,   with no constant term,   of the form $p_k(x)=\sum^{k}_{l=0}c_{k,l}x^l$. By  induction, we get that 
$\left|c_{k,l}\right|\leq k^{k}(m+1)^{k},$ for $  k\geq 1,\ 0\leq l\leq k.$ 
This leads to the uniform  estimate 
  with respect to $m$,
%indepently from $m$,
$$\left|p_k(x)\right|\asymp m^k \left|x\right|^k, \ \left|x\right|\rightarrow \infty, \ m\geq 1. $$
On another hand, termwise integration of (\ref{asymp ML m>1/2}) yields asymptotics of the primitive
\begin{equation*}\label{E-1}
	E_{\frac{1}{m},\frac{1}{m}}^{(-1)}(z)=\int^{z}_{0}E_{\frac{1}{m},\frac{1}{m}}(u)du,\ z\in\C.
\end{equation*}
 For any nonnegative integer $k$,  we  then obtain the following asymptotics, 
\begin{equation}\label{asymp deri ML m > 1/2}
	E^{(k-1)}_{\frac{1}{m},\frac{1}{m}}(z)=\begin{cases}
	 \frac{p_k(z^m)}{z^k} e^{z^m}+\eta_k(z),\ \left|\arg z\right|\leq\frac{\pi}{2m},\\
	\eta_k(z),\ \frac{\pi}{2m}<|\arg z|<\pi,
	\end{cases} 
\end{equation}
  %and  
	%the fact that $x\Gamma(x)\sim 1$ as $ x\rightarrow 0$, 
for $z\in\C\setminus\left\{0\right\}$. An inspection of the remainder shows that
	\begin{equation*}\label{remainder}
	\left|\eta_k(z)\right|\lesssim \begin{cases} m^{-1}\left|z\right|^{-k}\text{ if }k\geq 1,\\
	 m^{-1}\left|z\right|\text{ if }k=0,
	\end{cases}
\end{equation*}
	and  the function $E^{(k-1)}_{\frac{1}{m},\frac{1}{m}}(z)$ is bounded for $|z|\leq R_0<1$,  the constants   being independent of $m$ (see \cite{Bo,BaEr,WoZh}).  \\
 
 %and
%\begin{equation}\label{asymp deri ML m leq 1/2}
	%E^{(k-1)}_{\frac{1}{m},\frac{1}{m}}(z)=\sum^{N}_{j=-N}\frac{p_k(z^m e^{2\pi ijm})}{z^k}z^{m-1}e^{2\pi ij(m-1)}e^{z^m e^{2\pi ijm}}+O\left(z^{-k}\right),\ -\pi<\arg z\leq\pi,
	%\end{equation}

%for $0 <m <\frac{1}{2}$, $N<\frac{1}{2m}\leq N+1$. 

%In order to handle integrals involving the function $E$ and its derivatives,  we will use the Laplace method.
Let us recall the following result, known as the Laplace method \cite{FeEv}.

\begin{lemma}\label{laplace} 
Let  $a$ be a real number, and $f$,  $g$ be $\mathcal{C}^{\infty}$ real-valued functions on an interval $[a, b)$. For positive  real numbers $\lambda$, we consider the integrals
\begin{equation*}\label{int laplace }
		J(\lambda):=\int^{b}_{a}f(x)e^{\lambda g(x)}dx.
	\end{equation*}.
\begin{enumerate}
	\item [(a)] If $g$ attains its maximum at a unique interior point $x_0\in (a, b)$ and $g''(x_0)\neq 0,$
	\begin{equation}\label{laplace simple}
		J(\lambda)\asymp f(x_0)e^{\lambda g(x_0)}\sqrt{\frac{-2\pi}{ \lambda g''(x_0)}}\text{ as }\lambda\rightarrow +\infty.
	\end{equation}
	
\item [(b)] If $g$ attains its
maximum at the boundary point $x_0 =a$ and $g'(a)\neq 0$,
\begin{equation}\label{laplace dk}
J(\lambda)\asymp e^{\lambda g(a)}\sum^{+\infty}_{k=0} c_k \lambda^{-k-1},
\end{equation}
where
$$ c_k=\left(\frac{1}{-g'(x)}\frac{d}{dx}\right)^k\left.\frac{f(x)}{-g'(x)}\right|_{x=a}.$$
\end{enumerate}
\end{lemma}

In view of (\ref{asymp deri ML m > 1/2}), we set $I_m=\left[-\frac{\pi}{2m},\frac{\pi}{2m}\right]$, 
%\begin{equation}\label{def kappa}
%\kappa(t)=\begin{cases}
%1, \ |t|\geq 1\text{ and } \arg t \in I_m,\\
%0 \text{ otherwise},
%\end{cases}
%\end{equation}
and for any fixed $z\in\C^d$, $R>0$, 
\begin{align}\label{def SR}
	S_{R}&=S_{R}(z):=\left\{\zeta\in\C^d,\ \left|\left\langle z,\zeta\right\rangle\right|\geq R \text{ and } \arg\left\langle z,\zeta\right\rangle\in I_m \right\}.
\end{align}
 Moreover,  $\chi_R$ will denote the characteristic function of the closed ball in $\C^d$,
 centered at $0$, of radius $R$, and $d\sigma(\zeta)$ is the normalized Lebesgue measure on $\s^{2d-1}$ the unit sphere in $\C^d$.
 
The following Lemmas  supply estimates of  quantities depending on the parameter $m\geq 1$ and the variable $z\in\C^d$.
%involving the functions $E^{(k-1)}_{\frac{1}{m},\frac{1}{m}}$,  and we shall stress the dependence on $m$.
%The next technical Lemmas will be used to estimate integrals involving the functions $E^{(k-1)}_{\frac{1}{m},\frac{1}{m}}$. 

\begin{lemma}\label{int sphere}
Assume that $R_0, R>0$,  $r>0$, $p\geq 1$ and $l$ is a nonnegative integer. For $c$  a real number and 
 $z\in\C^d$,   put
\begin{align*}
%\label{Q def}
Q_{c,l,p,\beta}(R,r,z)&:=\int_{\s^{2d-1}\cap S_{R}(z)}\left|\left\langle z,r\zeta\right\rangle\right|^c\left|E^{(l-1)}_{\frac{1}{m},\frac{1}{m}}\left(\left\langle\beta^{1/m} z,r\zeta\right\rangle\right)\right|^pd\sigma(\zeta),\\
 Q^e_{c,l,p,\beta}(R,r,z)&:=\int_{\s^{2d-1}\setminus S_{R}(z)}\left|\left\langle z,r\zeta\right\rangle\right|^c\left|E^{(l-1)}_{\frac{1}{m},\frac{1}{m}}\left(\left\langle\beta^{1/m} z,r\zeta\right\rangle\right)\right|^pd\sigma(\zeta).\nonumber
\end{align*}

%where $d\sigma$ denotes the surface measure on the unit sphere $\s^{2d-1}$ of $\C^d$ and
When $\left|z\right|\geq R_0$, we have
%$r|z|\rightarrow\infty$, we have
\begin{align}
\label{asymp Q}
Q_{c,l,p,\beta}(R,r,z)&\asymp m^{pl-d}\left(r|z|\right)^{pl(m-1)-\frac{m}{2}+c-m(d-1)}e^{p\beta\left(r|z|\right)^m},\\
Q^e_{c,l,p,\beta}(R,r,z)&\lesssim\left(r|z|\right)^c.\nonumber
\end{align}
%where $a_Q(m,p,l,\beta)=m^{pl-d}\beta^{pl\left(1-\frac{1}{m}\right)+\frac{1}{2}-d}.$

\end{lemma}

\begin{proof}
  The proof being similar to that of Lemma 15 in \cite{BEY},  we only give a sketch. 
%On another hand, we will perform the computation up to some constant involving the dimension $d$, which is immaterial for our purpose. 
Since the measure $d\sigma$ is unitarily invariant, we may assume  that $z = (|z|,0,\cdots,0)$. Thus, 
$$ Q:=Q_{c,l,p,\beta}(R,r,z)=\int_{\s^{2d-1}}\kappa_R\left( |z|r\zeta_1\right)\left||z|r\zeta_1\right|^c\left|E^{(l-1)}_{\frac{1}{m},\frac{1}{m}}\left(\beta^{1/m}|z|r\zeta_1\right)\right|^pd\sigma(\zeta), $$
where $\kappa_R$ is the indicator function of the set $\Sigma_R=\left\{\xi\in\C^d,\ \left|\xi\right|\geq R, \ \arg \xi\in I_m \right\}$.
Assume $d\geq 2$.  Up to a constant factor, $Q$ is equal to  
$$\int^{1}_{0}\left||z|r \rho \right|^c\int^{\pi}_{-\pi}\kappa_R\left( |z|r \rho e^{i\theta}\right)\left|E^{(l-1)}_{\frac{1}{m},\frac{1}{m}}\left(\beta^{1/m}|z|r \rho e^{i\theta}\right)\right|^p d\theta\left(1-\rho^2\right)^{d-2}\rho  d\rho.$$
  For  $y=|z|r \rho$, the interior integral has the form
$$ J(y):=\int^{\pi}_{-\pi}\kappa_R\left( y e^{i\theta}\right)\left|E^{(l-1)}_{\frac{1}{m},\frac{1}{m}}\left(\beta^{1/m}y e^{i\theta}\right)\right|^p d\theta.$$
%Since $x\Gamma(x)\sim 1$ as $x\rightarrow 0$, the integral corresponding to $\pi/2m<\left|\theta\right|  <\pi$  is 
%comparable to $m^{-1} R^{-l-1}$ 
%bounded, independently of $m$ (Wong Zhao, erdelyi).
 Then, by (\ref{asymp deri ML m > 1/2}) and (\ref{laplace simple}), we get
\begin{align*}
J(y)&\asymp m^{pl}(1-\chi_R(y))\int^{\pi/2m}_{-\pi/2m}\left(\beta^{1/m}y\right)^{pl(m-1)} e^{\beta p y^m \cos(m\theta)}d\theta,\\
&\asymp m^{pl-1}(1-\chi_R(y))y^{pl(m-1)-\frac{m}{2}}e^{\beta p y^m}  .
\end{align*}
We now use the  change of variables $\rho=1-x$ and (\ref{laplace dk}), and
% with $\left[a,b\right]= \left[0,1\right],\  \lambda =p \beta r^m|z|^m,\  g(x) = (1 - x)^m,$
  observe that $$f (x) =
(1 - x)^{(m-1)pl-m/2+c}
(1- x)(2 - x)^{d-2}x^{d-2}$$ vanishes at $x_0=0$ to order $d-2$.  Letting $r|z|$ tend to $\infty$, we get (\ref{asymp Q}), which also holds for $d=1$.\\
The estimate for $Q^e_{c,l,p,\beta}(R,r,z)$ follows from the fact that $E^{(l-1)}_{\frac{1}{m},\frac{1}{m}}(z)$ is bounded outside $\Sigma_R$.

\end{proof}

\begin{lemma}\label{int blp beta}
Assume that $R,R_0 >0$,  $r>0$,  $p\geq 1$, and $l$ is a nonnegative integer. For $c$  a real number  and 
 $z\in\C^d$,   put
$$ I(z)= I_{c,l,p,\beta,\alpha}(R,z):=\int_{S_{R}(z)}\left|\left\langle z,\xi\right\rangle\right|^c\left|E^{(l-1)}_{\frac{1}{m},\frac{1}{m}}\left(\alpha^{1/m}\left\langle z,\xi\right\rangle\right)\right|^pe^{-\beta\left|\xi\right|^{2m}}dv(\xi),$$
and if  $c>-2d$,
$$ I^e(z)= I^e_{c,l,p,\beta,\alpha}(R,z):=\int_{\C^d\setminus S_{R}(z)}\left|\left\langle z,\xi\right\rangle\right|^c\left|E^{(l-1)}_{\frac{1}{m},\frac{1}{m}}\left(\alpha^{1/m}\left\langle z,\xi\right\rangle\right)\right|^pe^{-\beta\left|\xi\right|^{2m}}dv(\xi).$$
Then for $\left|z\right|\geq R_0,$
\begin{align}
\label{asymp I}
  I(z)&\asymp m^{pl-d-1}\left|z\right|^{2\left(pl-d\right)\left(m-1\right)+2c}e^{\frac{p^2\alpha^2}{4\beta}\left|z\right|^{2m}},\\
I^c(z)&\lesssim	\left|z\right|^b.\nonumber
\end{align}
	  
\end{lemma}

\begin{proof}
%\gamma=\gamma(m)=\frac{b}{m}+pl\left(1-\frac{1}{m}\right)-\frac{1}{2}-d+\frac{2d}{m}.
%Since $x\Gamma(x)\sim 1$ as $x\rightarrow 0$, 
%comparable to $m^{-1} R^{-l-1}$ 
%bounded, independently of $m$ (Wong Zhao, erdelyi).
%In $I(z)$, the contribution from $\xi\in \C^d\setminus S_m$ is bounded, independently of $m$.
  %Hence, 
	Using spherical coordinates, Lemma    \ref{int sphere}, the change of variable $t=r^m$,  and setting $\lambda=\frac{\alpha p}{2\beta}|z|^m$,
%$g(t)=-\beta\left(t-\lambda\right)^2$, 
$f(t)=t^{2\frac{d}{m}-1+pl\left(1-\frac{1}{m}\right)-\frac{1}{2}+\frac{c}{m}-(d-1)}$,   we get
\begin{align*}
I(z)&\asymp\int^{+\infty}_{0}r^{2d-1}Q_{c,l,p,\alpha}(R,r,z)e^{-\beta r^{2m}}dr\\
&\asymp m^{-1} m^{pl-d} |z|^{pl(m-1)-\frac{m}{2}+c-m(d-1)}
e^{\frac{\alpha^2 p^2}{4\beta}|z|^{2m}}\int^{+\infty}_{\frac{1}{\lambda}}f(x)f(\lambda)e^{-\beta\lambda^2 (x-1)^2}\lambda dx.
\end{align*}

Relation (\ref{laplace simple}) with $g(x)=-\beta(x-1)^2 $ and $x_0=1$ gives (\ref{asymp I}). The estimate for $I^e(z)$ follows from the estimate for $Q^e_{c,l,p,\beta}(R,r,z)$ in Lemma \ref{int sphere}.

\end{proof}

\begin{lemma}\label{estimate E in L infty}
Let $\alpha$ be a positive real number, $c$ a nonnegative real number, and $l$ a positive integer. 
%For $z\in\C^d$, set
%$$ T_{m,b,l,\alpha}(z):=\sup_{\zeta\in\C^d}\left|\left\langle z,\zeta \right\rangle\right|^b \left|E^{(l-1)}_{\frac{1}{m},\frac{1}{m}}\left(\alpha^{\frac{1}{m}}\left\langle %z,\zeta \right\rangle\right)\right|e^{-\frac{\alpha}{2}|\zeta|^{2m}}.$$
As $|z|\rightarrow\infty$, we have 
%$$T_{m,b,l,\alpha}(z)\lesssim m^l |z|^{2b+2l(m-1)} e^{\frac{\alpha}{2}|z|^{2m}}.$$
$$\sup_{\zeta\in\C^d}\left|\left\langle z,\zeta \right\rangle\right|^c \left|E^{(l-1)}_{\frac{1}{m},\frac{1}{m}}\left(\alpha^{\frac{1}{m}}\left\langle z,\zeta \right\rangle\right)\right|e^{-\frac{\alpha}{2}|\zeta|^{2m}}\lesssim m^l |z|^{2c+2l(m-1)} e^{\frac{\alpha}{2}|z|^{2m}}.  $$

\end{lemma}

\begin{proof}
From (\ref{asymp deri ML m > 1/2}), we deduce that, when $|z|\rightarrow\infty$,
$$ \left|\left\langle z,\zeta \right\rangle\right|^c \left|E^{(l-1)}_{\frac{1}{m},\frac{1}{m}}\left(\alpha^{\frac{1}{m}}\left\langle z,\zeta \right\rangle\right)\right|e^{-\frac{\alpha}{2}|\zeta|^{2m}} \lesssim m^l |z|^{c+l(m-1) }e^{\frac{\alpha}{2}|z|^{2m}}u\left(\left|\zeta\right|\right),$$
where 
$$ u(t)=t^{c+l(m-1) }e^{-\frac{\alpha}{2}\left(t^m-|z|^{m}\right)^2},\ t\in \R_+.$$
We conclude by estimating the maximum value of $u$.

\end{proof} 
 Straightforward consequences of  Lemmas \ref{int blp beta}  and  \ref{estimate E in L infty}    are estimates of the norms of the reproducing kernels
$$ K_{m,\alpha,z}=K_{m,\alpha}(.,z),\ z\in \C^d .$$    When $\left|z\right|\rightarrow\infty$,
%\begin{prop}\label
%for $ z\in\C^d  $,
\begin{align}\label{estimate kern in F}
\left\|K_{m,\alpha,z}\right\|_{F^{p}_{m,\alpha}}&\asymp  m^{d-\frac{d+1}{p}} \left|z\right|^{2d\left(1-\frac{1}{p}\right)(m-1)}e^{\frac{\alpha}{2}\left|z\right|^{2m}},\ 1\leq p<\infty,\\
\left\|K_{m,\alpha,z}\right\|_{F^{\infty}_{m,\alpha}}&\lesssim  m^{d} \left|z\right|^{2d(m-1)}e^{\frac{\alpha}{2}\left|z\right|^{2m}}.
	%\left\|K_{m,\alpha,z}\right\|_{F^{\infty}_{m,\alpha}}&\lesssim m^d \left|z\right|^{2d(m-1)}e^{\frac{\alpha}{2}\left|z\right|^{2m}}.
	 \end{align}

%\end{prop} 

We now establish pointwise estimates for functions in $F^{p}_{m,\beta}(Y)$.
\begin{prop}\label{pointwise Fp}
For $p>0$,  set
\begin{equation*}
%\label{tau p}
	\tau_p:=\frac{2}{p}d(m-1).
\end{equation*}
 Then  for any $f\in F^{p}_{m,\beta}(Y)$,
$$\left\| f(z)\right\|_Y\leq C\left\|f\right\|_{ F^{p}_{m,\beta}(Y)}\left(1+\left|z\right|\right)^{\tau_p}e^{\frac{\beta}{2}\left|z\right|^{2m}},\ z\in\C^d, $$
where the constant $C$ does not depend on $m$.
\end{prop}

\begin{proof}
For $z\neq 0$,  $B(z,r)$ denotes the ball of center $z$ and radius $r \left|z\right|^{1-m}$, with  euclidean volume $\left|B(z,r)\right|$. For large $|z|$,   the holomorphic function $w \mapsto K_{m,\beta}(w,z)$ does not vanish on
$B(z, r)$ by (\ref{asymp deri ML m > 1/2}). Thus the function $w \mapsto \left\|f(w)\right\|_{Y}^p|K_{m,\beta}	(w,z)|^{-p}$ is subharmonic in $B(z, r)$.
 Indeed if  $g:\C^d\rightarrow Y$ is  holomorphic,  the function $\zeta\mapsto \left\|g(\zeta)\right\|_Y$ is subharmonic, being the pointwise supremum of a family of subharmonic functions
$$\left\|g(\zeta)\right\|_Y=\sup_{y'\in Y^*\\
\left\|y'\right\|_{Y^*}=1}\left|\left\langle g(\zeta),y'\right\rangle_{Y, Y^*}\right|.$$
A straightforward calculation shows that,  if $|z|\rightarrow\infty$,   and  $w\in B(z, r)$, we have $ \left|w\right|\asymp\left|z\right|$ and 
$$\left|K_{m,\beta}(w,z)\right|^2\asymp K_{m,\beta}(z,z)K_{m,\beta}(w,w),$$
  (see also \cite{SY}).   Therefore,
\begin{align*}
\left\|f(z)\right\|_{Y}^p|K_{m,\beta}	(z,z)|^{-p}&\leq \frac{1}{\left|B(z, r)\right|}\int_{B(z, r)}\left\|f(w)\right\|_{Y}^p|K_{m,\beta}	(w,z)|^{-p}dv(w)\\
&\lesssim \left|z\right|^{2d(m-1)}\int_{B(z, r)}\left\|f(w)\right\|_{Y}^p|K_{m,\beta}	(w,w)K_{m,\beta}	(z,z)|^{-p/2}dv(w).
\end{align*}
Relation (\ref{asymp deri ML m > 1/2}) provides
% and the fact that for any $w\in B(z, r)$, we have $ \left|w\right|\asymp\left|z\right|$,
 %$$
%\left\|f(z)\right\|_{Y}^p \left(e^{\beta\left|z\right|^{2m}}\left|z\right|^{2d(m-1)}\right)^{-p/2}\lesssim \left|z\right|^{2d(m-1)}\int_{B(z, r)}\left\|f(w)\right\|_{Y}^p \left(e^{\beta\left|w\right|^{2m}}\left|w\right|^{2d(m-1)}\right)^{-p/2}dv(w),
%$$
%which gives 
the desired estimate.
\end{proof}

It follows from  Proposition \ref{pointwise Fp} that $F^{p}_{m,\beta}(Y)$ is a Banach space.\\

%In order to establish density results, which  are  known for the scalar Segal-Bargmann space  \cite{Zf},
% For a $Y$-valued entire function $f(\zeta)=\sum_{\nu\in \N^d} \hat{f}_{\nu} \zeta^{\nu}$, and a nonnegative integer $n$, $T_nf$ will denote the $n$-th Taylor polynomial of $f$, that is 
%\begin{align}\label{n th Taylor}
%	T_nf(\zeta)&=\sum_{\left|\nu\right|\leq n}\frac{1}{\nu!}\frac{\partial^{\nu}f}{\partial \zeta^{\nu}}\left(0\right)\zeta^{\nu}\\
%	&=\sum_{\left|\nu\right|\leq n}\hat{f}_{\nu} \zeta^{\nu}.
%\end{align}
 Now, recall the definition of 
the surface area integral means of a function $f$
\begin{equation*}
%\label{Mfr}
	 M^{p}_{p}(f,r):=\int_{\mathbb{S}^{2d-1}}\left\|f(r\zeta)\right\|^p_Yd\sigma(\zeta),\  0<p<\infty,\  r>0,
\end{equation*}
$$ M_{\infty}(f,r):=\sup_{|z|=r}\left\|f(z)\right\|_Y,$$
and the dilates of $f$, defined by  $f_r(\zeta)=f(r\zeta)$, $0<r<1$. The next result is known for the scalar Segal-Bargmann space (see \cite{Zf}), and its proof is standard; however, we give it for the sake of completeness.
%In our proof, $C$ is a  constant that may change from line to line.

\begin{prop}\label{density dilations }
Let $\beta>0$ , $1\leq p< \infty$ and let $Y$ be a Banach space.

\begin{enumerate}
	\item [(a)] If  $f$ is in $F^{p}_{m,\beta}(Y)$, then $\left\|f-f_r\right\|_{F^{p}_{m,\beta}(Y)}$ tends to $0$ as $r\rightarrow 1^-.$\\
	\item [(b)] If  $f$ is in $F^{\infty,0}_{m,\beta}(Y)$, then $\left\|f-f_r\right\|_{F^{\infty}_{m,\beta}(Y)}$ tends to $0$ as $r\rightarrow 1^-.$\\.
\end{enumerate}
\end{prop}

\begin{proof}
%Set ?
%$$ c=2d c_{m,\beta}  $$
 By integration in spherical coordinates, $\left\|f-f_r\right\|^{p}_{F^{p}_{m,\beta}(Y)}$ is a constant multiple of
$$\int^{+\infty}_{0}s^{2d-1} M^{p}_{p}(f-f_r,s) e^{-\frac{\beta}{2}ps^{2m}}ds. $$ 
The function $ r\mapsto M^{p}_{p}(g,r)$ is
 increasing,
 %( a mettre? see Krantz p 104 l 2.1.17), 
thus 
$$ M^{p}_{p}(f-f_r,s)\leq 2^p M^{p}_{p}(f,s). $$
Since  for any $s>0$,
$$\lim_{r\rightarrow 1^-}M^{p}_{p}(f-f_r,s)=0,$$
the dominated convergence theorem gives (a).

In order to prove (b), take $f\in F^{\infty,0}_{m,\beta}(Y)$,   $\epsilon>0$ and $0< r_0< 1$. We have
$$\sup_{|z|>R_0 r_0}\left\|f(z)\right\|_{Y} e^{-\frac{\beta}{2}\left|z\right|^{2m}} \leq\frac{\epsilon}{4} $$
for some $R_0>0$. 
Since $f_r $ converges uniformly to $f$ on compact sets,  there exists $r_1<1$ such that 
$$ \sup_{|z|\leq R_0 }\left\|f(z)-f_r(z)\right\|_{Y}e^{-\frac{\beta}{2}\left|z\right|^{2m}}\leq \frac{\epsilon}{2}\ \text{, whenever } r_1<r<1.$$
Besides, if $|z|>R_0$ and $1>r>\max(r_0,r_1)$, we have 
$$\left\|f(z)-f_r(z)\right\|_{Y} e^{-\frac{\beta}{2}\left|z\right|^{2m}} \leq\left\|f(z)\right\|_{Y} e^{-\frac{\beta}{2}\left|z\right|^{2m}}+\left\|f_r(z)\right\|_{Y} e^{-\frac{\beta}{2}\left|z\right|^{2m}}\leq \frac{\epsilon}{2}.$$
%& \leq \frac{\epsilon}{4}+\sup_{|z|>R_0 }\left\|f(rz)\right\|_{Y} e^{-\frac{\beta}{2}\left|z\right|^{2m}}\\
%& \leq \frac{\epsilon}{4}+\sup_{|z|>R_0 }\left\|f(rz)\right\|_{Y} e^{-\frac{\beta}{2}\left|rz\right|^{2m}}\\
%& \leq \frac{\epsilon}{4}+\sup_{|z|>R_0 r_0}\left\|f(z)\right\|_{Y} e^{-\frac{\beta}{2}\left|z\right|^{2m}}\\
%& \leq \frac{\epsilon}{2},
Now (b) follows from the inequality
$$\left\|f-f_r\right\|_{F^{\infty}_{m,\beta}(Y)}\leq  \sup_{|z|\leq R_0 }\left\|f(z)-f_r(z)\right\|_{Y}e^{-\frac{\beta}{2}\left|z\right|^{2m}}+\sup_{|z|> R_0 }\left\|f(z)-f_r(z)\right\|_{Y}e^{-\frac{\beta}{2}\left|z\right|^{2m}}.$$\\
\end{proof}
When $\HH$ is a Hilbert space, the orthogonality of monomials implies that every $f$ in $F^{2}_{m,\beta}(\HH)$ can be approximated by its Taylor polynomials.
For a general Banach space $Y$, the next proposition shows that  holomorphic polynomials are dense in $F^{p}_{m,\beta}(Y)$ for $1\leq p< \infty$, as well as  in $F^{\infty,0}_{m,\beta}(Y)$. The proof relies on convolutions with Fej\'er kernels and  may be known to specialists, but we include it for the sake of completeness.

\begin{prop}\label{density pol in F}
   Let $F\in\left\{F^{p}_{m,\beta}(Y), \ F^{\infty,0}_{m,\beta}(Y) \right\}$, for $1\leq p< \infty$. 
   If  $f$ is in $F$, there exists a sequence of holomorphic $Y$-valued  polynomials $(p_n)_n$ such that
	$ \left\|f-p_n\right\|_{F}$ tends to $0$ as $n\rightarrow +\infty.$

\end{prop}
\begin{proof}
%We will approximate a function $f\in F$ by  . 
If $\left(\theta_1,\cdots,\theta_d \right)\in \mathbb{R}^d$, consider the unitary linear transformation in $\mathbb{C}^d$ defined by
$R_{\theta}(z):=\left(e^{i\theta_1}z_1,\cdots,e^{i\theta_d}z_d\right)  $, for all $z=\left(z_1,\cdots,z_d\right)\in \C^d$. 
The torus $\T^d=\left\{\left(e^{i\theta_1},\cdots,e^{i\theta_d}\right),\ \left(\theta_1,\cdots,\theta_d\right)\in\left[-\pi,\pi\right]^d\right\}$ is equipped with the Haar measure $d\theta$, and, for a nonnegative integer $N$,  the Fej\'er kernel $F_N$ (see \cite{Gra}) is given by
\begin{equation*}
%\label{Fejer}
F_N\left(e^{i\theta_1},\cdots,e^{i\theta_d}\right):=\sum_{\left|m_j\right|\leq N,\\
m\in \Z}\left(1-\frac{\left|m_1\right|}{N+1}\right)\cdots \left(1-\frac{\left|m_d\right|}{N+1}\right)e^{i m\cdot \theta},
\end{equation*}
where $m\cdot \theta=m_1\theta_1+\cdots+m_d\theta_d $.
  %($F_N$ is an approximation of the identity on  $\T^d$). 
	The convolution 
\begin{equation}\label{convolution}
f_N(z)=\int_{\T^d}f\left( R_{-\theta}z\right)F_N\left(e^{i\theta_1},\cdots,e^{i\theta_d}\right)d\theta,\ z\in\C^d,
\end{equation}
is then a holomorphic polynomial with coefficients in $Y$, which belongs to $F$, and
   we have
$$  f_N-f=\int_{\T^d}(f \circ R_{-\theta}-f) F_N\left(e^{i\theta_1},\cdots,e^{i\theta_d}\right)d\theta. $$
Our aim is to prove that
\begin{equation}
\label{lim poly}
	\lim_{N\rightarrow\infty}\left\|f_N-f\right\|_F=0.
\end{equation}
To do so, we first observe that 
\begin{eqnarray*}
%\label{fejer00}
 \left\|f_N-f\right\|_F&\leq &\left(\int_{ V_\delta }+ \int_{\T^d\setminus V_\delta }\right)\left\|f \circ R_{-\theta}-f\right\|_F F_N\left(e^{i\theta_1},\cdots,e^{i\theta_d}\right)d\theta,
%&\leq &  \nonumber
\end{eqnarray*}
where $V_\delta$  $(\delta>0)$ denotes the set
%neighborhood of $0$ given by
$V_\delta:=\left\{\left(\theta_1,\cdots,\theta_d\right)\ :\ \ \left|\theta_j\right|\leq \delta,\, 1\le j\le d\right\}. $
 It is then enough to show that 
\begin{eqnarray}\label{fejer01}
\lim_{\delta\rightarrow 0} \sup_{\theta\in V_\delta} \left\|f \circ R_{-\theta}-f\right\|_F=0.
\end{eqnarray}
Indeed, suppose (\ref{fejer01}) holds and take $\varepsilon>0$. 
Due to the rotation invariance, we obtain
$$ \left\|f_N-f\right\|_F\leq \sup_{\theta\in V_\delta} \left\|f \circ R_{-\theta}-f\right\|_F\int_{ V_\delta }  F_N d\theta+
 2\left\|f\right\|_F\int_{\T^d\setminus V_\delta } F_N d\theta.
 $$
%By the  properties of the approximate identity $F_N$,  
One can choose $  \delta>0$ such that 
\begin{eqnarray*}
%\label{fejer02}
\sup_{\theta\in V_\delta} \left\|f \circ R_{-\theta}-f\right\|_F\leq \varepsilon,
\end{eqnarray*}
and an integer $N_0$ such that 
\begin{eqnarray*}
%\label{fejer03}
 \int_{\T^d\setminus V_\delta }F_N\left(e^{i\theta_1},\cdots,e^{i\theta_d}\right)d\theta\leq\varepsilon\ \text{ for }N\geq N_0,
\end{eqnarray*}
which implies (\ref{lim poly}).

It remains to show (\ref{fejer01}), which derives from the dominated convergence theorem when $F=F^{p}_{m,\beta}(Y)$.\\
 Now, let  $f$ be in  $F=F^{\infty,0}_{m,\beta}(Y)$, 
%\begin{align*}
%\left\|f \circ R_{-\theta}-f\right\|_F \leq  \left(\sup_{\left|z\right|\leq R}+\sup_{\left|z\right|> R}\right)\left\|\left[f \circ R_{-\theta}(z)-f(z) \right]e^{-\frac{\beta}{2}\left|z\right|^{2m}}\right\|_Y.
%\end{align*}
and $\varepsilon>0$.  
%\begin{eqnarray}\label{fejer05}
%\sup_{\left|z\right|> R}\left\|f(z) e^{-\frac{\beta}{2}\left|z\right|^{2m}}\right\|_Y\leq \varepsilon,
%\end{eqnarray}
%and by 
  By the  rotation invariance, there exists $R>0$ such that
\begin{eqnarray*}
%\label{fejer06}
\sup_{\left|z\right|> R}\left\|\left[f \circ R_{-\theta}(z)-f(z) \right]e^{-\frac{\beta}{2}\left|z\right|^{2m}}\right\|_Y\leq \varepsilon,
\end{eqnarray*}
uniformly in $\theta$. Since $f$ is uniformly continuous   on the compact set $\left\{z\in\C^d,\ \left|z\right|\leq R\right\}$,
% we have 
%\begin{eqnarray}\label{fejer06}
%\lim_{\theta\rightarrow 0}\sup_{\left|z\right|\leq R}\left\|\left[f \circ R_{-\theta}(z)-f(z) \right]e^{-\frac{\beta}{2}\left|z\right|^{2m}}\right\|_Y=0,
%\end{eqnarray}
%and thus 
we can choose $\delta>0$ such that for every $\theta\in V_\delta$ 
\begin{eqnarray*}
%\label{fejer07}
\sup_{\left|z\right|\leq R}\left\|\left[f \circ R_{-\theta}(z)-f(z) \right]e^{-\frac{\beta}{2}\left|z\right|^{2m}}\right\|_Y \leq \varepsilon.
\end{eqnarray*}
It follows that, whenever $\theta\in V_\delta$ ,
\begin{eqnarray*}
%\label{fejer08}
%\left\|f \circ R_{-\theta}-f\right\|_F = 
 \sup_{z\in\C^d}\left\|\left[f \circ R_{-\theta}(z)-f(z) \right]e^{-\frac{\beta}{2}\left|z\right|^{2m}}\right\|_Y\leq 2\varepsilon,
\end{eqnarray*}
which gives (\ref{fejer01}), and completes the proof.

\end{proof}

For  a Banach space $Y$,  we  consider the mapping $P_{\beta}: L^{p}_{m,\beta}( Y)\rightarrow F^{p}_{m,\beta}( Y)$, defined by
\begin{equation*}
%\label{def P beta}
P_{\beta}f(z):=\int_{\C^d}f(\zeta)K_{m,\beta}(z,\zeta)d\mu_{m,\beta}(\zeta),\ z\in \C^d.
\end{equation*}
 The  case when $m=1$ and $Y=\C$ has been studied in \cite{JPR,DZ}. However,  due to the lack of an appropriate group  of automorphisms  when $m\neq 1$, we need another approach for the spaces $F^{p}_{m,\alpha}(Y)$.\\

\begin{prop}\label{reproducing ppty}
Let $1\leq p\leq \infty$. Then $P_{\beta}: L^{p}_{m,\beta}( Y)\rightarrow F^{p}_{m,\beta}( Y)$ is bounded, and for  $f$ in $F^{p}_{m,\beta}(Y)$, the following reproducing formula holds
\begin{equation}\label{reproducing formula}
P_{\beta}f=f.
	%f(z)=P_{\beta}f(z)=\int_{\C^d}f(\zeta)K_{m,\beta}(z,\zeta)d\mu_{m,\beta}(\zeta),\ z\in \C^d.
\end{equation}
\end{prop}

\begin{proof}
The boundedness of  $P_{\beta}$ on $L^{p}_{m,\beta}( Y)$ is a consequence of  the estimates (\ref{estimate kern in F}) for the kernel. Indeed, the arguments of Theorem 1 in \cite{BEY} and the remark thereafter remain valid when $Y$ is a Banach space.
For $f$ in $L^{p}_{m,\beta}( Y)$, the analyticity of $P_{m,\beta}f$ is deduced from the scalar case  via linear functionals.

The reproducing formula is well known in the Hilbert case $F^{2}_{m,\beta}(\HH)$. For $Y$  a Banach space, we first consider the case when $f(\zeta)=\sum_{\nu\in\N^d}\hat f_{\nu}\zeta^{\nu}$ is in $ F^{2}_{m,\beta}(Y)$.
% Using that
%$$ P_{\beta}f(z)=\int_{\C^d}f(\zeta)\sum_{\nu\in\N^d} \frac{z^{\nu}\overline \zeta^{\nu}}{s_{\beta,\nu}}  d\mu_{m,\beta}(\zeta), $$
%Since
%$$ \frac{\left|z^{\nu}\right|}{s_{\beta,\nu}}\int_{\C^d}  \left\|f(\zeta)\right\|_Y\left|\zeta^{\nu}\right| d\mu_{m,\beta}(\zeta)\leq \frac{\left|z^{\nu}\right|}{s_{\beta,\nu}}\left\|f\right\|_{F^{2}_{m,\beta}(\C^d,Y)}s_{\beta,\nu}^{1/2}=\frac{\left|z^{\nu}\right|}{s_{\beta,\nu}^{1/2}}\left\|f\right\|_{F^{2}_{m,\beta}(\C^d,Y)},$$
 %Stirling's formula shows that the series
%$$ \sum \frac{\left|z^{\nu}\right|}{s_{\beta,\nu}}\int_{\C^d}  \left\|f(\zeta)\right\|_Y\left|\zeta^{\nu}\right| d\mu_{m,\beta}(\zeta)$$
%is convergent and thus by 
If  $K_{m,\beta}(z,\zeta)$ is replaced by its expansion, the dominated convergence theorem  yields
\begin{equation}\label{P beta f}
	 P_{\beta}f(z)=\sum_{\nu\in\N^d} \frac{z^{\nu}}{s_{\beta,\nu}}\int_{\C^d}f(\zeta) \overline \zeta^{\nu}  d\mu_{m,\beta}(\zeta),\ z\in\C^d.
\end{equation}
%Integration in polar coordinates gives
%\begin{align*}
%\int_{\C^d}f(\zeta) \overline \zeta^{\nu}  d\mu_{m,\beta}(\zeta)&=\int^{+\infty}_{0}2dr^{2d-1 +\left|\nu\right|}\int_{\mathbb{S}^{2d-1}}f(r\xi)\overline \xi^{\nu}d\sigma(\xi)e^{-\beta r^{2m}}dr.\\
%&=\int^{+\infty}_{0}2r^{2d-1 +\left|\nu\right|}\int_{\mathbb{S}^{2d-1}}\sum_{\nu'\in\N^d}\hat f_{\nu'}r^{\nu'}\xi^{\nu'}\overline \xi^{\nu}d\sigma(\xi)e^{-\beta r^{2m}}dr\\
%\end{align*}
%For each $\nu \in\N^d$, the uniform convergence of the power series of $f$ on compact sets gives
%\begin{align*}
%\int_{\mathbb{S}^{2d-1}}f(r\xi)\overline \xi^{\nu}d\sigma(\xi)&=\int_{\mathbb{S}^{2d-1}}\sum_{\nu'\in\N^d}\hat f_{\nu'}r^{\left|\nu'\right|}\xi^{\nu'}\overline \xi^{\nu}d\sigma(\xi)\\
%&=\sum_{\nu'\in\N^d}\hat f_{\nu'}r^{\left|\nu'\right|}\int_{\mathbb{S}^{2d-1}}\xi^{\nu'}\overline \xi^{\nu}d\sigma(\xi)\\
%&=\hat f_{\nu}r^{\left|\nu\right|}\int_{\mathbb{S}^{2d-1}} \left|\xi^{\nu}\right|^2 d\sigma(\xi)\\
%\end{align*}
%and it follows  that formula \ref{reproducing formula} is true for all $f\in F^{2}_{m,\beta}(\C^d,Y)$.
Using the uniform convergence of the power series of $f$ on compact sets, and integrating in spherical coordinates, we get that, for $R>0$,
\begin{align*}
\int_{\left|\zeta\right|\leq R}f(\zeta) \overline \zeta^{\nu}  d\mu_{m,\beta}(\zeta)&=\sum_{\nu'\in\N^d}\hat f_{\nu'}\int_{\left|\zeta\right|\leq R}\zeta^{\nu'}  \overline \zeta^{\nu}  d\mu_{m,\beta}(\zeta)
=\hat f_{\nu}\int_{\left|\zeta\right|\leq R}\left| \zeta^{\nu} \right|^2 d\mu_{m,\beta}(\zeta).
\end{align*}
Taking limit as $R\rightarrow+\infty$, and combining with   (\ref{P beta f}), we see that $ P_{\beta}f(z)=f(z)$.

 Now if $f$ is in $F^{p}_{m,\beta}(Y)$, the pointwise estimate of Proposition \ref{pointwise Fp}
%$$\left\|f(\zeta)\right\|_Y\lesssim \left\|f\right\|_{F^{p}_{m,\beta}(Y)}|\zeta|^{\frac{2}{p}d(m-1)}e^{\frac{\beta}{2}\left|\zeta\right|^{2m}},\ \left|\zeta\right|\rightarrow\infty,$$
shows that for $0<r<1$, the dilation $f_r$  is in $F^{2}_{m,\beta}(Y)$. We obtain  (\ref{reproducing formula}) by applying the dominated convergence Theorem and taking limit as $r\rightarrow 1$ in the  formula $P_{\beta}f_r(z)=f_r(z)$. 

\end{proof}

%The next Lemma will be used to characterize the dual space of $F^{p}_{m,\beta}(\C^d,Y)$.
Given   a scalar or $\be(X)$-valued entire function $u$, and $x\in X$, one  defines the $X$-valued entire function
\begin{equation*}
%\label{tenso fct}
(u\otimes x)(z):= u(z)x,\ z\in\C^d.
\end{equation*}
We now identify the dual space of $F^{p}_{m,\beta}(Y)$, $1\leq p<\infty$,  and of $F^{\infty,0}_{m,\beta}(Y)$.
\begin{prop}\label{dual Fp 1 leq p}
 Let $Y$ be a Banach space, and $Y^*$ its dual space.
\begin{enumerate}
	\item [(a)] For $ 1<p<\infty$, the dual space of $F^{p}_{m,\beta}(Y)$ is isometrically isomorphic to $F^{p'}_{m,\beta}(Y^*)$.
 \item [(b)] The dual space of $F^{1}_{m,\beta}(Y)$ is isometrically isomorphic to $F^{\infty}_{m,\beta}(Y^*)$.
\item  [(c)] The dual space of $F^{\infty,0}_{m,\beta}(Y)$ is isometrically isomorphic to $F^{1}_{m,\beta}(Y^*)$.
\end{enumerate}

\end{prop}
%,  under the pairing
 %\begin{equation}\label{pairing beta}
%\left\langle g,f\right\rangle_{\beta}=\int_{\C^d}\left\langle g(\overline{\zeta}), f(\zeta)\right\rangle_{Y, Y^*} d\mu_{m,\beta}(\zeta),
%\end{equation}
%where $f\in F^{p'}_{m,\beta}(\C^d,Y^*)$ and $g\in F^{p}_{m,\beta}(\C^d,Y).$

\begin{proof}Our proof is inpired of \cite{AB,Co}. The notation  $\left\langle .,.\right\rangle_{Y-Y^*}$  is shortened into $\left\langle .,.\right\rangle$.
 First assume that $ 1<p<\infty$. By H\"older's inequality, the linear operator $M$ defined  by
\begin{equation}
\label{def M}
Mf(g):=\int_{\C^d}\left\langle  g(\zeta),f(\overline{\zeta})\right\rangle d\mu_{m,\beta}(\zeta),\ f\in F^{p'}_{m,\beta}(Y^*),\ g\in F^{p}_{m,\beta}(Y), 
\end{equation}
  is bounded from  $F^{p'}_{m,\beta}(Y^*)$ to $\left(F^{p}_{m,\beta}(Y)\right)^*$ and $\left\|M\right\|\leq 1.$ \\
	If  $f$ is in $F^{p'}_{m,\beta}(Y^*)$ and $Mf=0$, we apply $Mf$ to $K_{m,\beta,\overline z}\otimes x$ for any $z\in \C^d$ and $x\in Y$, and   the reproducing property implies that $f=0$, which shows  that $M$ is injective.\\
		%the reproducing property enhances that for any $z\in \C^d$ and $x\in Y$, 
%	$$ 0=Jf\left(K_{m,\beta,\overline z}\otimes x\right)= \left\langle  x,f(z)\right\rangle,$$
%which shows that $f=0$.
In order to prove the surjectivity, take $\Phi$ a bounded linear functional on $F^{p}_{m,\beta}(Y)$, and define the $ Y^*$-valued entire function $\phi$  by
\begin{equation*}
%\label{def phi}
\left\langle  x ,\phi({z}) \right\rangle=\Phi\left( K_{m,\beta, \overline z}\otimes x\right),\ x\in Y,\ z\in \C^d .
\end{equation*}
Our aim is to show that $\phi$ is in $L^{p'}_{m,\beta}(Y^*)$  and that $\Phi=M\phi$. To do so, we set $\phi_1(\zeta)=\phi(\zeta)c_{m,\beta}e^{-\frac{\beta}{2}\left|\zeta\right|^{2m}} $, and  consider the integrals 
\begin{equation*}
%\label{def phi1}
\Phi_1(h)=\int_{\C^d}\left\langle  h(\zeta),\phi_1( \overline\zeta)\right\rangle dv(\zeta),\ h\in L^p(\C^d,dv,Y).
\end{equation*}
   For  any  continuous compactly supported $Y$-valued function  $h$,  we see that
$$  \Phi_1(h)
%=\int_{\C^d}\left\langle h(\zeta),\phi( \overline\zeta)\right\rangle c_{m,\beta}e^{-\frac{\beta}{2}\left|\zeta\right|^{2m}} dv(\zeta)
=\Phi\left(\int_{\C^d}K_{m,\beta,\zeta}\otimes  h_1(\zeta) d\mu_{m,\beta}(\zeta)\right)= \Phi P_{\beta}h_1,$$
where   $h_1(\zeta)=h(\zeta)e^{\frac{\beta}{2}\left|\zeta\right|^{2m}}$. 
%\begin{align*}
%\Phi_1(h)&=\int_{\C^d}\left\langle h(\zeta),\phi( \overline\zeta)\right\rangle c_{m,\beta}e^{-\frac{\beta}{2}\left|\zeta\right|^{2m}} dv(\zeta)\\
%&=\int_{\C^d}\Phi\left(K_{m,\beta,\zeta}\otimes  h(\zeta)\right)e^{\frac{\beta}{2}\left|\zeta\right|^{2m}} d\mu_{m,\beta}(\zeta)\\
%&=\int_{\C^d}\Phi\left(K_{m,\beta,\zeta}\otimes  h_1(\zeta)\right) d\mu_{m,\beta}(\zeta)\\
%&=\Phi\left(\int_{\C^d}K_{m,\beta,\zeta}\otimes  h_1(\zeta) d\mu_{m,\beta}(\zeta)\right)\\
%\end{align*}
%since **** for each $\zeta$, $K_{m,\beta,\zeta  }\otimes  h_1(\zeta) $ is in $F^{p}_{m,\beta}(\C^d,Y)$ and the integral $\int_{\C^d}K_{m,\beta,\zeta}\otimes  h_1(\zeta) d\mu_{m,\beta}(\zeta)= P_{\beta}h_1$ is also in $F^{p}_{m,\beta}(\C^d,Y)$ 
%because  ****
%$$\left\| K_{m,\beta,z}\otimes  h_1(\zeta) \right\|_{F^{p}_{m,\beta}(\C^d,Y)}\lesssim \left\|h_1(\zeta) \right\|_{Y}\left|\zeta\right| ^{2\frac{d}{p'}(m-1)} e^{\frac{\beta}{2}\left|\zeta\right|^{2m}}$$
%and
%$h_1$ is compactly supported, and $\Phi$ is bounded on $F^{p}_{m,\beta}(\C^d,Y)$ (Dinculeanu p 123). 
%or
%$ \Phi_1(h)$
%Thus $\Phi_1$ defines a bounded linear functional on the space $L^p(\C^d,dv,Y)$.
The boundedness of $P_{\beta}$ on $L^{p}_{m,\beta}(Y)$ yields that $\Phi_1$ is bounded on $L^p(\C^d,dv,Y)$
%$L^{p}_{m,\beta}(Y)$
 with $\left\|\Phi_1\right\|\leq\left\|\Phi\right\|\left\|P_{\beta}\right\|$, and we deduce that $\phi_1\in L^{p'}(\C^d,dv,Y^*)$.
%, and thus $\phi_1$ is in $L^{p'}(\C^d,dv,Y^*)$ by an approximation argument.\\
%Due to the boundedness of $P_{\beta}$ on $L^{p}_{m,\beta}(\C^d,Y)$ we get
%$$ \left|\Phi_1(h)\right|\leq\left\|\Phi\right\| \left\|P_{\beta}\right\|\left\|h_1\right\|_{L^{p}_{m,\beta}(\C^d,Y)}=\left\|\Phi\right\| \left\|P_{\beta}\right\|\left\|h\right\|_{L^{p}(\C^d,Y)};$$
%thus $\Phi_1$ is bounded on $L^{p}_{m,\beta}(\C^d,Y)$ with $\left\|\Phi_1\right\|\leq\left\|\Phi\right\|\left\|P_{\beta}\right\|$. 
%This  implies that $\phi_1$ is in $L^{p'}(\C^d,Y^*)$ by an approximation argument (dinc).\\
Indeed,  for every integer $n$, let $\chi_n$ be  the characteristic function of the closed ball of $\C^d$ of center $0$ and  radius $n$; we observe that
$$  \Phi_1(h\chi_n)=\int_{\C^d}\left\langle  h(\zeta),\left(\phi_1\chi_n\right)\left(\overline \zeta\right) \right\rangle dv(\zeta),\ h\in L^{p}(\C^d,Y).$$
Recall that  the $L^{p'}$-norm (see \cite{Din} Theorem 6,  chap. 2) is given by
 %$ \phi_1\chi_n $ induces a bounded linear functional on $L^{p}(\C^d,Y)$, and
$$ \left\|\phi_1\chi_n\right\|_{L^{p'}(\C^d,Y^*)}=\sup_{\left\|h\right\|_{L^{p}(\C^d,Y)}=1}\left|\int_{\C^d}\left\langle  h(\zeta),\left(\phi_1\chi_n\right)\left(\overline \zeta\right) \right\rangle dv(\zeta)\right| \leq\left\|\Phi_1\right\|.$$
%$$\leq\left\|\Phi\right\|\left\|P_{\beta}\right\|$$
Now Beppo-Levi's Theorem implies that $\phi_1\in L^{p'}(\C^d,dv,Y^*)$, or equivalently $\phi\in L^{p'}_{m,\beta}(Y^*)$.\\
 It remains to prove that , for any $f\in F^{p}_{m,\beta}(Y)$,
\begin{equation*}
%\label{j phi f}
	\Phi(f)=M\phi(f).
\end{equation*}
  Considering  the dilate $f_r$, $0<r<1$, and  $R>0$, we split the integral
\begin{align*}
M\phi(f_r)
%&=\int_{\C^d}\left\langle f_r(\zeta),\phi(\overline\zeta)\right\rangle d\mu_{m,\beta}(\zeta)\\
&=\int_{\C^d}\Phi\left(K_{m,\beta,\zeta}\otimes f_r(\zeta)\right)d\mu_{m,\beta}(\zeta)\\
&=\int_{\left|\zeta\right|\leq R}\Phi\left(K_{m,\beta,\zeta}\otimes f_r(\zeta)\right)d\mu_{m,\beta}(\zeta)+\int_{\left|\zeta\right|> R}\Phi\left(K_{m,\beta,\zeta}\otimes f_r(\zeta)\right)d\mu_{m,\beta}(\zeta).\\
\end{align*}
 Since  the first integral is equal to $\Phi P_{\beta}\left( \chi_R f_r\right)$, and  $\lim_{R\rightarrow \infty}\left\|\left(1-\chi_R\right)f_r\right\|_{F^{p}_{m,\beta}(Y)}=0$,  the boundedness of $\Phi P_{\beta}$ on $L^{p}_{m,\beta}(Y)$ ensures that
$$ \lim_{R\rightarrow \infty} \int_{\left|\zeta\right|\leq R}\Phi\left(K_{m,\beta,\zeta}\otimes f_r(\zeta)\right)d\mu_{m,\beta}(\zeta)=\Phi P_{\beta}\left(  f_r\right)=\Phi \left(  f_r\right).$$
%\begin{align*}
  %\int_{\left|\zeta\right|\leq R}\Phi\left(K_{m,\beta,\zeta}\otimes f_r(\zeta)\right)d\mu_{m,\beta}(\zeta)
	%&=\Phi\left(\int_{\left|\zeta\right|\leq R}K_{m,\beta,\zeta}\otimes f_r(\zeta)d\mu_{m,\beta}(\zeta)\right)\\
	%&=\Phi\left( \int_{\C^d}K_{m,\beta,\zeta}\otimes\chi_R (\zeta) f_r(\zeta)d\mu_{m,\beta}(\zeta)\right)\\
	%&=
	%\end{align*}
	%and since $$\lim_{R\rightarrow \infty}\left\|\left(1-\chi_R\right)f_r\right\|_{F^{p}_{m,\beta}(Y)}=0,$$
	%the boundedness of $\Phi P_{\beta}$ on $L^{p}_{m,\beta}(Y)$ implies that
	%$$ \lim_{R\rightarrow \infty} \int_{\left|\zeta\right|\leq R}\Phi\left(K_{m,\beta,\zeta}\otimes f_r(\zeta)\right)d\mu_{m,\beta}(\zeta)=\Phi P_{\beta}\left(  f_r\right)=\Phi \left(  f_r\right).$$
On another hand, the estimates established in (\ref{estimate kern in F}) and  Proposition \ref{pointwise Fp} give
\begin{align*}
\left|\int_{\left|\zeta\right|> R}\Phi\left(K_{m,\beta,\zeta}\otimes f_r(\zeta)\right)d\mu_{m,\beta}(\zeta)\right|
%&\leq 
%\int_{\left|\zeta\right|> R}\left\|\Phi\right\|\left\|K_{m,\beta,\zeta}\otimes f_r(\zeta)\right\|_{F^{p}_{m,\beta}(\C^d,Y)}d\mu_{m,\beta}(\zeta)\\
&\leq\int_{\left|\zeta\right|> R}\left\|\Phi\right\|\left\|K_{m,\beta,\zeta}\right\|_{F^{p}_{m,\beta}}\left\|f_r(\zeta)\right\|_Y d\mu_{m,\beta}(\zeta)\\
%&\lesssim \left\|\Phi\right\|  \left\|f\right\|_{F^{p}_{m,\beta}(\C^d,Y)}\int_{\left|\zeta\right|> R}\left|\zeta\right|^{2\frac{d}{p'}(m-1)}e^{\frac{\beta}{2}\left|\zeta\right|^{2m}}\left|\zeta\right|^{2\frac{d}{p}(m-1)}e^{\frac{\beta}{2}r^{2m}\left|\zeta\right|^{2m}}d\mu_{m,\beta}(\zeta)\\
&\lesssim \left\|\Phi\right\| \left\|f\right\|_{F^{p}_{m,\beta}(Y)} \int_{\left|\zeta\right|> R}\left|\zeta\right|^{2d(m-1)}e^{-\frac{\beta}{2}(1-r^{2m})\left|\zeta\right|^{2m}}dv(\zeta),\\
\end{align*}
which tends to $0$ as $R\rightarrow \infty.$ We have proved that 
$ \Phi(f_r)=M\phi(f_r) $, for $0<r<1$. Taking limit as $r\rightarrow 1$, we obtain that   $M$ is surjective. 

The cases $F^{1}_{m,\beta}(Y)$ and $F^{\infty,0}_{m,\beta}(Y)$ are handled similarly.

%\par

%it is sufficient to prove that for any $f\in \mathcal{L}^{p}_{m,\beta}(\C^d,Y)$, 
%\begin{equation}\label{phi in Lp'}
%\left|\int_{\C^d}\left\langle f(z), \phi(\overline{z}) \right\rangle_{Y, Y^*}d\mu_{m,\beta}(z)\right|\leq C \left\|f\right\|_{\mathcal{L}^{p}_{m,\beta}(\C^d,Y)}
%\end{equation}
%for some constant $C$. First we observe that by Lemma \ref{int blp beta} we get
%\begin{align*}
%\left|\left\langle f(z), \phi(\overline{z}) \right\rangle_{Y, Y^*}\right|&\leq\left\|\Phi\right\|\left\|f(z)K_{m,\beta,z}\right\|_{F^{p}_{m,\beta}(\C^d,Y)}\\
%&=\left\|\Phi\right\|\left\|f(z)\right\|_Y\left\|K_{m,\beta,z}\right\|_{F^{p}_{m,\beta}(\C^d,\C)}\\
%&\lesssim \left\|\Phi\right\|\left\|f(z)\right\|_Y\left|z\right|^{\frac{2}{p'}d(m-1)}e^{\frac{\beta}{2}\left|z\right|^{2m}},\\
%\end{align*}
%which ensures the absolute convergence of the integral in \ref{phi in Lp'}. Now, using Holder inequality, we get
%\begin{align*}
%\int_{\C^d}\left|\left\langle f(z), \phi(\overline{z}) \right\rangle_{Y, Y^*}\right|d\mu_{m,\beta}(z)&\lesssim %\left\|\Phi\right\|\int_{\C^d}\left\|f(z)\right\|_Y\left|z\right|^{\frac{2}{p'}d(m-1)}e^{\frac{\beta}{2}\left|z\right|^{2m}}d\mu_{m,\beta}(z)\\
%&\lesssim \left\|\Phi\right\|\left\|f\right\|_{\mathcal{L}^{p}_{m,\beta}(\C^d,Y)},
%\end{align*}
%which proves that
%$$  \left\|\phi\right\|_{F^{p'}_{m,\beta}(\C^d,Y^*)}\lesssim \left\|\Phi\right\|.$$

\end{proof}

%When the target space is a Hilbert space $\HH$,
 Proposition \ref{density kernel} proves another density result for $F^{p}_{m,\beta}(Y )$, which will be  used in section \ref{compact}.

\begin{prop}\label{density kernel}
 The set $V$ of functions of the form 
$$ h(z)=\sum^{n}_{l=1}K_{m,\beta}(z,a_l)c_l, \qquad a_l\in\C^d,\ c_l\in Y,\ n\text{ integer},$$
is dense in $F$, if $F\in \left\{F^{p}_{m,\beta}(Y),F^{\infty,0}_{m,\beta}(Y)\right\} $, $1\leq p$.

\end{prop}

\begin{proof}
We  apply the Hahn-Banach Theorem to the space
$$ V^{\perp}=\left\{ h\in F^{  *},\ Mh(g)=0, \text{ for every }g\in V\right\},$$
 where $M $ is defined in  (\ref{def M}), namely
$$Mh(g):=\int_{\C^d}\left\langle  g(\zeta),f(\overline{\zeta})\right\rangle d\mu_{m,\beta}(\zeta).  $$
 If $h\in V^{\perp}$, we have $Mh( K_{m,\beta,\overline a}\otimes c)=\left\langle c,h(a)\right\rangle=0$ for every $a\in \C^d,\ c\in Y$. Hence $ V^{\perp}$  reduces to $\left\{0\right\}$, which completes the proof.
\end{proof}
\section{Boundedness of Hankel operators}\label{boundedness}

In this section, we characterize the boundedness of the Hankel operator $h_b$ and study the dependence of $\left\|h_b\right\|$ on the parameter $m\geq 1$. \\
When $b$ is in $L^{\infty}_{m,\alpha}( \mathcal{B}(\HH))$, a straightforward computation shows that $h_b$ and $h_{P_{\alpha}b}$ coincide on the set of holomorphic polynomials, which is dense in $F^{p}_{m,\alpha}( \HH)   $, $1\leq p<\infty$. 
 Thus, we may assume that  $b$ is in $F^{\infty}_{m,\alpha}( \mathcal{B}(\HH))$, and in such case $h_b$ is densely defined. \\
% Recall that from (\ref{sigma}),
%$$ \sigma_p=\max\left(1,\frac{d+1}{p}\right). $$

\begin{lemma}\label{hb bd implies Rdb p}
%Let $\alpha_0, $ be a positive numbers and $\alpha \geq \alpha_0>0$.
Suppose that $b$ is in $F^{\infty}_{m,\alpha}( \mathcal{B}(X))$ and  that $h_b$ extends to a bounded linear operator on $F^{p}_{m,\alpha}( X)   $, $1\leq p<\infty$. Then 
\begin{equation*}
%\label{cond deri radial}
	\left\|\RR b\left(z\right)\right\|_{\mathcal{B}(X)}\leq C\left\|h_b\right\|\left( m^{d+1} \left|z\right|^{2m}+\left|z\right|\right)e^{\frac{\alpha}{4}\left|z\right|^{2m}},\ z\in\C^d,
\end{equation*}
for some constant $C$ which is independent of $m$.
\end{lemma}

\begin{proof}
Let  $0<R_0<1$ and $x_0\in X$. 
%The letter $C$ will denote a constant that may change from line to line,  but is independent of $b,m$ and $z$.
   We apply $h_b$ to the constant function $x_0$.   The reproducing property implies  that $h_b x_0=b\otimes x_0$, and 
$$(h_b x_0)(z)=b(z)x_0=\int_{\C^d}(h_b x_0)(\zeta) K_{m,\alpha}\left(z,\zeta\right)d\mu_{m,\alpha}(\zeta),\ z\in\C^d.$$
With the notation
\begin{equation*}%\label{def F}
	F(t):=tE^{(d)}_{\frac{1}{m},\frac{1}{m}}(t),\ t\in\C,
\end{equation*}
differentiation under the integral shows that 
\begin{equation}\label{calcul Rb}
	 \RR b(z)x_0= C_m c_{m,\alpha}\int_{\C^d}(h_bx_0)(\zeta)F\left(\left\langle \alpha^{1/m}z,\zeta\right\rangle\right) e^{-{\alpha}\left|\zeta\right|^{2m}}dv(\zeta).
\end{equation}
We shall estimate $\left\|\RR b(z)x_0\right\|_X$ for $\left|z\right|\leq R_0$ and $\left|z\right|\geq R_0$ separately.\\

	\textbf{step 1: $\left|z\right|\leq R_0$}.\\ 
  Definitions (\ref{dmu cm alpha}) imply $C_mc_{m,\alpha}\asymp m$. We assume that $1< p$, the case $p=1$ being similar. It follows from H\"older's inequality and the boundedness of $h_b$ that
$$ \left\|\RR b(z)x_0\right\|_X\lesssim m\left\|h_b\right\|  \left\|x_0\right\|_X I_m(z),$$
where
$$I_m(z)^{p'}:=\int_{\C^d}\left| F \left(\left\langle \alpha^{1/m}z,\zeta\right\rangle\right)\right|^{p'}e^{-p'\frac{\alpha}{2}\left|\zeta\right|^{2m}}dv(\zeta).$$
Since the power series of $F$ has positive coefficients,  the estimate (\ref{asymp deri ML m > 1/2}) implies that
$$  \left| F \left(\left\langle \alpha^{1/m}z,\zeta\right\rangle\right)\right|\lesssim \left|z\right|\left(1+ m^{d+1} \left|\zeta\right|\left(R_0 \left|\zeta\right| \right)^{(m-1)(d+1)}\right)e^{\alpha R_0^m\left|\zeta\right|^m}.$$
 Therefore
 \begin{align*}
 I_m(z)^{p'}
%&\lesssim\left|z\right|^{p'}\left[\int_{\left|\zeta\right|\leq 1}+\int_{\left|\zeta\right|> 1}\right]\left|\left|\zeta\right| E^{(d)}_{\frac{1}{m},\frac{1}{m}} \left( \alpha^{1/m}R_0\left|\zeta\right|\right)\right|^{p'}e^{-p'\frac{\alpha}{2}\left|\zeta\right|^{2m}}dv(\zeta)\\ 
%&\leq \left|z\right|^{p'}\int_{\C^d}\left|\left|\zeta\right|E^{(d)}_{\frac{1}{m},\frac{1}{m}} \left( \alpha^{1/m} R_0\left|\zeta\right|\right)\right|^{p'}e^{-p'\frac{\alpha}{2}\left|\zeta\right|^{2m}}dv(\zeta)\\
%&\lesssim c_{m,\alpha}\left|z\right|^{p'}\alpha^{p'/m}\int_{\C^d} \left|\left|\zeta\right|m^{d+1}\left( \alpha^{1/m} R_0\left|\zeta\right|\right)^{(d+1)(m-1)} e^{\alpha R_0^m\left|\zeta\right|^m} \right|^{p'}e^{-p'\frac{\alpha}{2}\left|\zeta\right|^{2m}}dv(\zeta)\\
&\lesssim \left|z\right|^{p'}\left(1+m^{p'(d+1)}R_0^{p'(d+1)(m-1)}J_m(z)\right),
%&\lesssim \left|z\right|^{p'}\left(1+m^{p'(d+1)}R_0^{p'(d+1)(m-1)}\int_{\C^d} \left|\zeta\right|^{p'\left(1+(d+1)(m-1)\right)}
 %e^{p'\alpha R_0^m\left|\zeta\right|^m}e^{-p'\frac{\alpha}{2}\left|\zeta\right|^{2m}}dv(\zeta)\right),
\end{align*} 
and the integral
$$ J_m(z):=\int_{\C^d} \left|\zeta\right|^{p'\left(1+(d+1)(m-1)\right)}
 e^{p'\alpha R_0^m\left|\zeta\right|^m}e^{-p'\frac{\alpha}{2}\left|\zeta\right|^{2m}}dv(\zeta) $$ 
is computed  via spherical coordinates and  a change of variable.  Thus,  for some constant $C$,
$$  J_m(z)=C{m}^{-1}e^{p'\frac{\alpha}{2}R^{2m}_{0}}\int^{+\infty}_{0}t^{\gamma_m} e^{-p'\frac{\alpha}{2}\left(t-R^{m}_{0}\right)^2} dt,$$
 the exponent $ \gamma_m:= 2\frac{d}{m}-1+\frac{p'}{m}\left(1+(d+1)(m-1)\right) $ being bounded above and below, independently of $m$. Hence,
 $mJ_m(z)$ is bounded, 
$I_m(z)^{p'}\lesssim\left|z\right|^{p'}$, and  
\begin{equation*}%\label{z leq R0}
	\left\|\RR b(z)x_0\right\|_X
%&\leq C_m  \left\|h_bx_0\right\|_{F^{p}_{m,\alpha}(X )}I_m(z)\\
\lesssim \left\|h_b\right\|  \left\|x_0\right\|_X\left|z\right|,\ \left|z\right|\leq R_0. 
\end{equation*}

\textbf{step 2: $\left|z\right|\geq R_0$}.\\
%Next we shall estimate $\left\|\RR b(z)\right\|_{\be(X)}$ for $\left|z\right|\geq R_0$.
 Let $y_0\in X^{*}$. From (\ref{calcul Rb}), we get
\begin{align*}%\label{diff1 rep form}
\left\langle \left(\RR b\right)(2^{1/m}z)x_0,y_0\right\rangle_{X-X^{*}}&=C_m   c_{m,\alpha}\int_{\C^d}\left\langle b(\zeta)x_0,y_0\right\rangle_{X-X^{*}}  F\left(\left\langle \left(2\alpha\right)^{1/m}z,\zeta\right\rangle\right)d\mu_{m,\alpha}(\zeta)\\
&=\int_{S_R} \cdots+ \int_{\C^d\setminus S_R}\cdots=I(z)+J(z),\nonumber
\end{align*}
%\begin{equation}\label{Jk tilde split}
%{J}(z)=\int_{S_m} \cdots+ \int_{\C^d\setminus S_m}\cdots= I+II.
%\end{equation}
where $S_R=S_R(z)$ was defined in (\ref{def SR}).\\

  In order to study $I(z)$, the integral over $S_R$,  we set
	$$t= \left\langle z,\zeta\right\rangle,\text{ and }t_{\beta}=\beta^{1/m}t, \  \beta>0, \text{  with }\arg t\in I_m\text{ and }|t|\geq R,$$
	and  establish asymptotics of the integrand when $|t|\rightarrow \infty$. \\
	Recall that (\ref{asymp deri ML m > 1/2}) implies that, for  $t\in\C\setminus \left\{0\right\}$, and $\left|\arg t\right|\leq\frac{\pi}{2m}$,
\begin{equation}\label{asym Ek-1}
	E^{(k-1)}_{\frac{1}{m},\frac{1}{m}}(t)=
	 \frac{p_k(t^m)}{t^k} e^{t^m}+O(1)m^{-1}\left|t\right|^{-k},\ \text{ for any integer }k\geq 1, 
\end{equation}
	where the polynomial $p_k$ has degree $k$ and that 
	\begin{equation}\label{asym E -1}
	E^{(-1)}_{\frac{1}{m},\frac{1}{m}}(t)=
	  e^{t^m}+O(1)m^{-1}\left|t\right|.
\end{equation}
   On another hand,  there exist real numbers $a_l$ such that
\begin{equation}\label{eucli divi p 1+d}
p_{1+d}(2T)=p_{1}(T)\sum^{d}_{l=0}a_{l}p_l(T).
\end{equation}
%because $p_{1+d}(0)=0$
A careful inspection of the coefficients show that  $ a_{d}=2^{1+d}$, and
\begin{gather*}
\left|a_{l}\right|\lesssim (m+1)^{d-l},\ 0\leq l\leq d.
\end{gather*}
%with a constant  which is independent of $m$.\\ 
%(it depends on $d$ only).\\
For an entire   function $f(t), \ t\in \C$,	and   an  integer $k\geq 0$, let  $T_kf$ denote the Taylor polynomial of $f$ of degree $k$ at $0$, and set $T_{-1}f=T_{-2}f=0$. For $0\leq l\leq d$, we  define the  entire function 
\begin{equation*}%\label{def Gl}
	G_l(t):=\frac{1}{t^{d-l-1}}\left[E^{(l-1)}(t)-T_{d-l-2}E^{(l-1)}(t)\right],\  t\in \C.
\end{equation*}
%With this notation, 
Using (\ref{asym Ek-1}), we derive   asymptotics of $F\left(t_{2\alpha}\right )$ in terms of $p_{d+1}$. Taking into account (\ref{eucli divi p 1+d}) and   asymptotics for $E_{\frac{1}{m},\frac{1}{m}}$ and $G_l$ given by (\ref{asym Ek-1}) and  (\ref{asym E -1}), we have 
\begin{equation}\label{F t 2 alpha}
F\left(t_{2\alpha}\right )=2^{-d/m}\left[\sum^{d}_{l=0}a_{l}A_l(t_{\alpha})\right]+O\left(1\right)t^{-d},
\end{equation}
where
\begin{align}\label{def Al}
A_l(t)&:=\frac{p_1\left(\alpha t^m\right)}{t_{\alpha}}   e^{\alpha t^m}\frac{1}{\left(t_{\alpha}\right)^{d-l-1}}\frac{p_l\left(\alpha t^m\right)}{\left(t_{\alpha}\right)^{l}}e^{\alpha t^m}\\
%&=\left[E_{\frac{1}{m},\frac{1}{m}}\left(t_{\alpha}\right)+O\left(1\right)m^{-1}t_{\alpha}^{-1}\right]\left[G_l(t_{\alpha})+ O\left(1\right)m^{-1}\right]\\\nonumber
%&=E_{\frac{1}{m},\frac{1}{m}}\left(t_{\alpha}\right)G_l(t_{\alpha})+O\left(1\right)m^{-1}\left[ t_{\alpha}^{-1}G_l(t_{\alpha})+E_{\frac{1}{m},\frac{1}{m}}\left(t_{\alpha}\right)+m^{-1}\right]\\
&=E_{\frac{1}{m},\frac{1}{m}}\left(t_{\alpha}\right)G_l(t_{\alpha})+O\left(1\right)m^{-1}\left[t_{\alpha}^{2+l-d}E_{\frac{1}{m},\frac{1}{m}}^{(l-1)}(t_{\alpha})+tE_{\frac{1}{m},\frac{1}{m}}\left(t_{\alpha}\right)+t^2\right]\nonumber.
%&=\left[E^{(1-1)}\left(t_{\alpha}\right)+O(1)t_{\alpha}^{-1}\right]\left[\frac{1}{\left(t_{\alpha}\right)^{d-l-1}}E^{(l-1)}\left(t_{\alpha}\right)+O(1 )t_{\alpha}^{1-d}\right],
\end{align}

 % has the same  asymptotics as $\frac{E^{(l-1)}(t)}{t^{d-l-1}}$, namely,  as $\left|t\right|\to \infty,\ \arg t\in I_m$,
%\begin{align}\label{asymp Gl}
%	G_l(t)&=\frac{E^{(l-1)}(t)}{t^{d-l-1}}+O\left(t\right),
%	\end{align}
Combining  (\ref{def Al}) and (\ref{F t 2 alpha}), one gets
$$ I(z)=C_m 2^{-d/m}\sum^{d}_{l=0}a_{l} \sum^{4}_{i=1}B_{l,i}, $$
where
%and, for $0\leq l\leq d$, there exist  bounded functions  $\epsilon_{l,i}(t)$, $i=2,3$ such that
%\begin{multline*}
%B_l=\int_{S_m}\left\langle b(\zeta)x_0,y_0\right\rangle_{X-X^{*}}\left[E\left(\alpha^{1/m}\left\langle z,\zeta\right\rangle\right)G_l(\alpha^{1/m}\left\langle z,\zeta\right\rangle)\right.\\
%\left. +O(1)\left(\left\langle z,\zeta\right\rangle^{l-d}E^{(l-1)}\left(\alpha^{1/m}\left\langle z,\zeta\right\rangle\right)+\left\langle z,\zeta\right\rangle E\left(\alpha^{1/m}\left\langle z,\zeta\right\rangle\right)\right)\right]\mm,
%\end{multline*}
 %Therefore 
%\begin{multline*}
\begin{align*}
B_{l,1}&=\int_{S_R}\left\langle b(\zeta)x_0,y_0\right\rangle_{X-X^{*}}   E\left(\alpha^{1/m} \left\langle z,\zeta\right\rangle \right)G_l\left(\alpha^{1/m}\left\langle z,\zeta\right\rangle\right)\mm,\\
B_{l,2}&=\int_{S_R}\epsilon_{l,2}\left(\alpha^{1/m}\left\langle z,\zeta\right\rangle\right)\left\langle (h_bx_0)(\zeta),y_0\right\rangle_{X-X^{*}} \left\langle z,\zeta\right\rangle^{2+l-d}  E^{(l-1)}\left(\alpha^{1/m}\left\langle z,\zeta\right\rangle\right)\mm,\\
B_{l,3}&=\int_{S_R}\epsilon_{l,3}\left(\alpha^{1/m}\left\langle z,\zeta\right\rangle\right)\left\langle (h_bx_0)(\zeta),y_0\right\rangle_{X-X^{*}} \left\langle z,\zeta\right\rangle  E\left(\alpha^{1/m}\left\langle z,\zeta\right\rangle\right)\mm,\\
B_{l,4}&=\int_{S_R}\epsilon_{l,4} \left(\alpha^{1/m}\left\langle z,\zeta\right\rangle\right)  \left\langle z,\zeta\right\rangle^2 \left\langle b(\zeta)x_0,y_0\right\rangle_{X-X^{*}} \mm,
\end{align*}
and   $\left|\epsilon_{l,i}(t)\right|\lesssim m^{-1}$, $i=2,3,4$. 
In order to estimate $\sum^{d}_{l=0}a_l B_{l,1}$, let us introduce the following test functions
\begin{gather}\label{test fct}
x_z(\zeta)=2^{-\frac{d}{m}}C_m E\left(\alpha^{1/m}\left\langle \zeta,\overline{z}\right\rangle\right)x_0,\\
y_{l,z}(\zeta)=G_l\left(\alpha^{1/m}\left\langle \zeta,z\right\rangle\right)y_0,\  0\leq l\leq d,\text{ and }
y_z= \sum^{d}_{l=0}a_{l} y_{l,z},
\end{gather}
for $\zeta,z\in\C^d$.
%and  we obtain 
%\begin{multline*}
%I=\left\langle h_b x_z,y_{z}\right\rangle_{\alpha} +O(1)   c_{m,\alpha}\left\|h_b \right\|_{F^{p}_{m,\alpha}(X)}\left\|x_0\right\|_X\left\|y_0\right\|_{X^*}\left|z\right|^{\frac{2}{p}d(m-1)}e^{\frac{\alpha}{2}\left|z\right|^{2m}}.
%\end{multline*}
We observe that
$$  \sum^{d}_{l=0}a_l B_{l,1}= \left\langle h_b x_z,y_{z}\right\rangle_{\alpha}-\int_{\C^d\setminus S_R}\left\langle b(\zeta)x_z(\overline \zeta),y_z(\zeta)\right\rangle\mm . $$
%$$ C_m I= \left\langle h_b x_z,y_{z}\right\rangle_{\alpha} +\sum^{d}_{l=0}a_l\left(B_{l,2}+B_{l,3}\right).$$
 First suppose $1<p<\infty$.  Lemma \ref{int blp beta}  
%and  \ref{estimate E in L infty}
%(with $b=0, l=d, p=1, \alpha=\beta/2$),
 yields
%\begin{gather*}
\begin{align}\label{esti xz}
\left\|x_z\right\|^p_{F^{p}_{m\alpha}(X)}&\asymp  \left\|x_0\right\|_X^p m^p m^{p-d-1}\left|z\right|^{2(p-d)(m-1)}e^{p\frac{\alpha}{2}\left|z\right|^{2m}},\\
\left\|a_l y_{l,z}\right\|^{p'}_{F^{p'}_{m\alpha}(X)}&\lesssim   \left\|y_0\right\|_{X^*}^{p'}m^{p'l-l-1}\left|z\right|^{2\left(p'l-d\right)\left(m-1\right)+2p'\left(l+1-d\right)}e^{p'\frac{\alpha}{2}\left|z\right|^{2m}}\ \text{ for }0\leq l\leq d.\nonumber
\end{align}
Hence
%\begin{align*}
%\left|a_{l}\right|   \left\|y_{l,z}\right\|_{F^{p'}_{m\alpha}(\HH)}&\asymp  \left\|y_0\right\|m^{d-\frac{d+1}{p'}}\left|z\right|^{2\left(l-\frac{d}{p'}\right)\left(m-1\right)+2\left(l+1-d\right)}e^{\frac{\alpha}{2}\left|z\right|^{2m}}\\
%&\lesssim\left|a_{d}\right|   \left\|y_{d,z}\right\|_{F^{p'}_{m\alpha}(\HH)},
%\end{align*}
\begin{equation}\label{esti yz}
	\left\|y_{z}\right\|_{F^{p'}_{m\alpha}(X)}\lesssim \left|a_{d}\right|   \left\|y_{d,z}\right\|_{F^{p'}_{m\alpha}(X)}\lesssim  \left\|y_0\right\|_{X^*}m^{d-\frac{d+1}{p'}}\left|z\right|^{2\frac{d}{p}\left(m-1\right)+2}e^{\frac{\alpha}{2}\left|z\right|^{2m}},
\end{equation}
and
one obtains that
%\ref{int blp beta}, we obtain
\begin{align*}
 \left|\left\langle h_b x_z,y_{z}\right\rangle_{\alpha}\right|
%&\leq\left\|h_b\right\| \left\|x_z\right\|_{F^{p}_{m\alpha}(\HH)}\left\|y_{z}\right\|_{F^{p'}_{m\alpha}(\HH)}+O(1)\left\|x_0\right\| \left\|y_0\right\|\left\|b\right\|_{F^{\infty}_{m\alpha}( \mathcal{B}(X) )}c_{m,\alpha}e^{\frac{\alpha}{2}\left|z\right|^{2m}}\\
&\lesssim m\left\|x_0\right\|_X\left\|y_0\right\|_{X^*}\left\|h_b\right\|\left|z\right|^{2m}e^{\alpha\left|z\right|^{2m}}.
\end{align*}
%We observe that by Lemma \ref{estimate E in L infty},
 %the estimates 
%for the norms of the test functions $x_z, y_z$ obtained when $1<p<\infty$ 
%still hold when $p\in\left\{1,\infty\right\}$, with the convention $\frac{1}{\infty}=0$.\\
Using (\ref{asym Ek-1}),  H\"{o}lder's inequality, the pointwise estimates from Proposition \ref{pointwise Fp}, the boundedness of $h_b$ and Lemma
\ref{int blp beta}, we see that the quantities
$$\left|\int_{\C^d\setminus S_R}\left\langle b(\zeta)x_z(\overline \zeta),y_z(\zeta)\right\rangle\mm\right|,\ \left| J(z)\right|\text{ and }C_m \left|a_l\left(B_{l,2}+B_{l,3}+B_{l,4}\right)\right|$$ are bounded by 
$$  Cm^{d+1}\left\|h_b\right\|\left\| x_0\right\|_X \left\|y_0\right\|_{X^{*}}e^{c\alpha\left|z\right|^{2m}},$$
where   $C$ and  $c$ are constants, with $1/2<c<1$. \\
%Due to the boundedness of $h_b$     on $F^{p}_{m\alpha}(X)$, we have
%$$ \left|C_m I\right|\leq C_m\left\|h_b\right\| \left\|x_z\right\|_{F^{p}_{m\alpha}(X)}\left\|y_{z}\right\|_{F^{p'}_{m\alpha}(X^*)}+Cm^{1+\frac{d+1}{p}}\left\|x_0\right\| \left\|y_0\right\|\left\|h_b\right\|\left|z\right|^{\frac{2}{p}d(m-1)}e^{\frac{\alpha}{2}\left|z\right|^{2m}}.$$
	We have shown that
\begin{equation*}%\label{esti R b2}
\left|\left\langle \RR b(2^{1/m}z)x_0,y_0\right\rangle_{X-X^*}\right|\lesssim
  \left\|x_0\right\|_X\left\|y_0\right\|_{X^*}\left\|h_b\right\|\left[m\left|z\right|^{2m}e^{\alpha\left|z\right|^{2m}}+m^{d+1}e^{c\alpha\left|z\right|^{2m}}\right],
\end{equation*}
for  $\left|z\right|\geq R_0$,  which completes the proof when $1<p<\infty$.

The argument for $p=1$ is similar.
\end{proof}

The estimate of $\mathcal R b$ will provide  the desired estimate for $b$.

\begin{lemma}\label{de Rdb à b}
Let $X$ be a Banach space,  $b$  in $F^{\infty}_{m,\alpha}(\mathcal{B}(X) )$  such that  $h_b$ is bounded on $F^{p}_{m,\alpha}( X)   $ ($1\leq p<\infty$) and
\begin{equation}\label{cond Rd b}
	\left\|\RR b\left(z\right)\right\|_{\mathcal{B}(X)}\leq C \left\|h_b\right\|\left(m^{d+1}\left|z\right|^{2m}+\left|z\right|\right)e^{\frac{\alpha}{4}\left|z\right|^{2m}},\ z\in\C^d,
\end{equation}
for some constant $C$.
Then $b\in F^{\infty}_{m,\frac{\alpha}{2}}( \mathcal{B}(X))$, and there exists a constant $C'$, independent of $m$, such that
 \begin{equation}\label{cond b}
	\left\| b\right\|_{ F^{\infty}_{m,\frac{\alpha}{2}}\left(\mathcal{B}(X)\right)}\leq C'm^{d} \left\|h_b\right\|.
\end{equation}
\end{lemma}

\begin{proof}
Let $z\in\C^d$ be fixed and set $a=\frac{\alpha}{4}\left|z\right|^{2m}$.
Integrating  (\ref{cond Rd b}), 
we get
\begin{align*}
\left\|b(z)-b(0)\right\|_{ \mathcal{B}(X) }&=\left\|\int^{1}_{0} \RR b\left(tz\right)\frac{dt}{t}\right\|_{ \mathcal{B}(X) }\\
%&=\left\|\int^{1}_{0} \RR b\left(tz\right)\frac{dt}{t}\right\|_{\mathcal{B}(X)}\\
&\leq    C \left\|h_b\right\| \left[m^{  d+1 }\int^{1}_{0}t^{2m-1}  \left|z\right|  ^{2m}e^{at^{2m}}dt+\int^{1}_{0}\left|z\right|e^{at^{2m}}dt\right].\\
\end{align*}
A direct computation shows that
$$ m^{ d+1}\int^{1}_{0}t^{2m-1}  \left|z\right|  ^{2m}e^{at^{2m}}dt\leq m^{ d}\frac{2}{\alpha}e^{\frac{\alpha}{4}\left|z\right|^{2m}}.$$ 
As for the other integral, we observe that if $\left|z\right|\geq 1$, 
\begin{align*}
\int^{1}_{0}\left|z\right|e^{at^{2m}}dt&\leq\int^{\frac{1}{\left|z\right|}}_{0}\left|z\right| e^{at^{2m}}dt +\int^{1}_{\frac{1}{\left|z\right|}} \left|z\right|\left(t\left|z\right|\right)^{2m-1}  e^{at^{2m}}dt\leq\left|z\right|e^{\frac{\alpha}{4}}+\frac{2}{m\alpha}e^{\frac{\alpha}{4}\left|z\right|^{2m}}\\
&\lesssim e^{\frac{\alpha}{4}\left|z\right|^{2m}},
\end{align*}
and that this estimate holds in fact for  any $z\in\C^d$.
 We have shown that
$$ \left\|b(z)-b(0)\right\|_{ \mathcal{B}(X) }\lesssim \left\|h_b\right\| m^{  d }e^{\frac{\alpha}{4}\left|z\right|^{2m}},$$
which, together with 
$$\left\|b(0)x_0\right\|_X =\left\|\int_{\C^d}\left(h_b x_0\right)(\zeta)d\mu_{m, \alpha}(\zeta)\right\|_X \lesssim\left\|h_b \right\|\left\| x_0\right\|_X,$$
 implies (\ref{cond b}).

\end{proof}

We are now ready to prove 	
 Proposition \ref{charact F infty}.   Set $\alpha=2\beta$.
\begin{proof}
[Proof of  Proposition \ref{charact F infty}]

\textbf{Equivalence $(a) \Leftrightarrow (b)$}.
If $b$ is in $F^{\infty}_{m,\frac{\alpha}{2}}(Y)$, the function $b(2^{1/m}\cdot)$ is in $F^{\infty}_{m, 2\alpha}(Y)$, and   $c(\zeta)=2^{\frac{d}{m}}b(2^{1/m}\zeta)e^{-\alpha\left|\zeta\right|^{2m}}$ is bounded on $\C^d$. From the reproducing formula, we see that, for $ z\in \C^d,$
\begin{align*}
	b(z)=b\left(2^{1/m}\frac{z}{2^{1/m}}\right)&=\int_{\C^d}b\left(2^{1/m}\zeta\right)K_{m,2 \alpha }(\frac{z}{2^{1/m}},\zeta)d\mu_{m,2 \alpha}(\zeta)=P_{\alpha}c(z).
	%&=\int_{\C^d}c(\zeta)K_{m,2 \alpha }(\frac{z}{2^{1/m}},\zeta)d\mu_{m, \alpha}(\zeta)\\
	%&=\int_{\C^d}c(\zeta)K_{m, \alpha }\left(z,\zeta\right)d\mu_{m, \alpha}(\zeta)\\
	\end{align*}
	Now   if $b=P_{\alpha}c$ for some  $c\in L^{\infty}\left(\C^d, Y\right)$, we derive from Lemma \ref{int blp beta} that
$$ \left\|b(z)\right\|_Y\leq\int_{\C^d}\left\|c\right\|_{L^{\infty}\left(\C^d, Y\right)}\left|K_{m, \alpha }\left(z,\zeta\right)\right|d\mu_{m, \alpha}(\zeta)
\lesssim \left\|c\right\|_{L^{\infty}\left(\C^d, Y\right)}e^{\frac{\alpha}{4}\left|z\right|^{2m}}. $$
%and the equivalence between $(1)$ and $(2)$ is proven.

 \textbf{Implication $(c) \Rightarrow (a)$}. We notice that, for any  positive integer $k$,  the definition of the radial derivative yields that  $\left\|\RR^k b(\zeta)\right\|_{Y}\lesssim \left|\zeta\right|$ when $\left|\zeta\right|\leq 1 $. Therefore
$$ \left\|\RR^k b(\zeta)\right\|_{Y}\lesssim\left(\left|\zeta\right|+\left|\zeta\right|^{2km}e^{\frac{\beta}{2}\left|\zeta\right|^{2m}}\right),\  \zeta\in\C^d.$$
 The implication is then obtained by  induction, using a similar argument to that of    the proof of Lemma \ref{de Rdb à b}.

 \textbf{Implication $(a) \Rightarrow (c)$}. Assume that
 $b\in F^{\infty}_{m, \beta}(Y)$.  Iterated differentiation of the reproducing formula in Proposition \ref{reproducing ppty} 
%$$ b(z)=\int_{\C^d}b(\zeta) K_{m, \beta}(z,\zeta)d\mu_{m, \beta}(\zeta)\ z\in\C^d.$$
%y differentiation, we get
shows that there are constants $\alpha_{k,l}$ such that
\begin{align*}%\label{R1 b}
\RR^{k} b(z)
%&=C_m\int_{\C^d}b(\zeta)\RR_z E^{(d-1)}_{\frac{1}{m},\frac{1}{m}}\left(\beta^{\frac{1}{m}}\left\langle z,\zeta\right\rangle\right)d\mu_{m, \beta}(\zeta)\\
&=C_m c_{m, \beta}\sum^{k}_{l=0}\alpha_{k,l}\int_{\C^d} b(\zeta)
\left\langle z,\zeta \right\rangle^{l}
 E^{(d-1+l)}_{\frac{1}{m},\frac{1}{m}}
\left(\beta^{\frac{1}{m}}\left\langle z,\zeta\right\rangle\right)e^{-\beta\left|\zeta\right|^{2m}}dv(\zeta).
\end{align*}
Since $\left\|b(\zeta)\right\|_Ye^{-\frac{\beta}{2}\left|\zeta\right|^{2m}}$ is bounded,
 the estimates from Lemma \ref{int blp beta} show that
\begin{align*}
\left\|\RR^k b(z)\right\|_{Y}
%&\leq C_m c_{m, \beta}\beta^{\frac{1}{m}}\int_{\C^d}\left\|b(\zeta)\right\|_Y\left|\left\langle z,\zeta\right\rangle\right|
%\left|E^{(d+1-1)}_{\frac{1}{m},\frac{1}{m}}
%\left(\beta^{\frac{1}{m}}\left\langle z,\zeta\right\rangle\right)\right|e^{-\beta\left|\zeta\right|^{2m}}dv(\zeta)\\
%&\lesssim m  \left\|b\right\| _{F^{\infty}_{m, \beta}(\C^d,Y)} \int_{\C^d}\left|\left\langle z,\zeta\right\rangle\right|\left|E^{(d+1-1)}_{\frac{1}{m},\frac{1}{m}}
%\left(\beta^{\frac{1}{m}}\left\langle z,\zeta\right\rangle\right)\right|e^{-\frac{\beta}{2}\left|\zeta\right|^{2m}}dv(\zeta)\\
%&\lesssim m \beta^{\frac{d+1}{m}} \left\|b\right\| _{F^{\infty}_{m, \beta}(\C^d,Y)} I_{m,1,d+1,1,\frac{\beta}{2},\beta}(z)\\
%&\lesssim m \beta^{\frac{d+1}{m}} \left\|b\right\| _{F^{\infty}_{m, \beta}(\C^d,Y)}\beta^{(d+1)\left(1-\frac{1}{m}\right)-d}\left|z\right|^{2(d+1-d)(m-1)+2}e^{\frac{\beta}{2}\left|z\right|^{2m}}\\
&\lesssim  \left\|b\right\| _{F^{\infty}_{m, \beta}(Y)}\left|z\right|^{2km}e^{\frac{\beta}{2}\left|z\right|^{2km}},\  \text{ as }\left|z\right|\rightarrow\infty.
\end{align*}
\end{proof}

From now,  $X=\HH$ is a separable Hilbert space. We aim at showing that if $b$ is in $F^{\infty}_{m, \frac{\alpha}{2}}(\BB(\HH))$,
 the operator $h_b$ is bounded on $F^{p}_{m,\alpha}(\mathcal{H})$.
Recall that the Hankel operator of symbol $T(b):\C^d\rightarrow\BB(\BB(\mathcal{H}))$  is 
defined by
\begin{equation*}%\label{def tb}
T(b)(z)S=b(z)S,\ S \in\BB(\mathcal{H}). 
\end{equation*}
%Our strategy to prove the implication .. in Theorem ... is inspired of \cite{Co} , where  is considered.
  % and that $G^p$ is defined by  \ref{Gp}.
%	denotes one of the spaces $F^{p}_{m,\alpha}(\C^d,\mathcal{S}^q(\mathcal{H}))$, where  $q\in\left\{p,2\right\}$.  
	For any $x,y$ in $\HH,$ define the rank one operator
$$ y\otimes x: \HH\rightarrow \HH,\ (y\otimes x)(h)=\left\langle h,x\right\rangle y,\ h\in\HH.$$
  Lemma \ref{Htb entraine hb} provides a sufficient condition for the boundedness of $h_b$ on $F^{p}_{m,\alpha}(\mathcal{H})$. 
Recall that 
   $G^p$ is one of the spaces $F^{p}_{m,\alpha}(\mathcal{S}^q(\mathcal{H}))$ , where $q\in\left\{p,2\right\}$, cf (\ref{Gp}). We  have seen in Proposition \ref{dual Fp 1 leq p} that the dual space $(G^p)^{*}$ is identified with  $G^{p'}$. In Lemmas \ref{Htb entraine hb} and \ref{hTb bd then hb bd}, the duality $G^p - (G^{p})^{*}$ is denoted by $\left\langle ., .\right\rangle_{G^p - (G^{p})^{*}}$, namely
	 \begin{equation*}
	%\label{Gp Gp'}
		 \left\langle  x,y\right\rangle_{G^p - (G^{p})^{*}}=\int_{\C^d}\left\langle  x({\zeta}) , y(\zeta)\right\rangle_{\text{tr}} d\mu_{m,\alpha}(\zeta),\ \text{for }x\in G^p\text{  and }y\in G^{p'}.
	 \end{equation*}
	
	\begin{lemma}\label{Htb entraine hb}
	Let $b\in F^{\infty}_{m,\alpha}(\mathcal{H})$.
	 Let $1\leq p<\infty$, and assume that $h_{T(b)}$  extends to a bounded operator on   $G^p$. Then  $h_b$ is bounded on $F^{p}_{m,\alpha}(\mathcal{H})$
and
$$ \left\|h_b\right\| \leq \left\|h_{T(b)}\right\|_{\be(G^p)}.  $$

\end{lemma} 

\begin{proof}
We combine ideas  from Lemma 2.3 in \cite{Co}, which handles the case of the Bergman space of the unit disc and $p=2$, with a duality argument.
Take $(e_n)_{n}$ an   orthonormal basis of $\HH,$ and
assume that  $h_{T(b)}$is  bounded  on $G^p$. For $x$ in $F^{p}_{m,\alpha}(\mathcal{H})$, $y$ in $F^{p'}_{m,\alpha}(\mathcal{H})$, and a fixed integer $n$, one has
\begin{align*}
\left\langle h_b x,y\right\rangle_{\alpha}&=\int_{\C^d}\left\langle \left(b(\zeta)x(\overline{\zeta})\right) \otimes e_n, y(\zeta)\otimes e_n\right\rangle_{\text{tr}} d\mu_{m,\alpha}(\zeta)\\
&=\int_{\C^d}\left\langle \left(T(b)(\zeta)x(\overline{\zeta}) \right)\otimes e_n, y(\zeta)\otimes e_n\right\rangle_{\text{tr}} d\mu_{m,\alpha}(\zeta)\\
&=\left\langle h_{T(b)} x\otimes e_n,y\otimes e_n\right\rangle_{G^{p}-\left(G^{p}\right)^*}.
%(\Sl^p)(\Sl^p')
\end{align*} 
Since the dual space of  $F^{p}_{m,\alpha}(\mathcal{H})$ is identified with  $F^{p'}_{m,\alpha}(\mathcal{H})$
by Proposition \ref{dual Fp 1 leq p},  we deduce that 
$$ \left\|h_b\right\| \leq \left\|h_{T(b)}\right\|_{\be(G^p)}.  $$
%The same proof holds when $h_{T(b)}$  extends to a bounded operator on   $F^{p}_{m,\alpha}(\C^d,\mathcal{S}^2(\mathcal{H}))$.

\end{proof}
We now give a sufficient condition for the boundedness of $h_{T(b)}$  on $G^{p}$.

\begin{lemma}\label{hTb bd then hb bd}
Suppose $1\leq p< \infty$ and $b\in  F^{\infty}_{m,\frac{\alpha}{2}}(\BB(\mathcal{H}))$. Then $h_{T(b)}$ is bounded on $G^{p}$
and
$$  \left\|h_{T(b)}\right\|_{\be(G^p)}\leq 2^{\frac{d}{m}} \left\|b\right\|_{F^{\infty}_{m,\frac{\alpha}{2}}(\BB(\mathcal{H}))}.$$

\end{lemma} 

\begin{proof}
%We shall only consider the case $F^{p}_{m,\alpha}(\C^d,\mathcal{S}^p(\mathcal{H}))$, the case of $F^{p}_{m,\alpha}(\C^d,\mathcal{S}^2(\mathcal{H}))$ can be repeated verbatim.
The proof of Proposition \ref{charact F infty} implies that $b=P_{\alpha}c$, where $c\left(\zeta2^{-1/m}\right)=2^{\frac{d}{m}}b(\zeta)e^{-\frac{\alpha}{4}\left|\zeta\right|^{2m}}$ is in $L^{\infty}\left( \mathcal{B}(\HH)\right)$.  Taking $x$ (resp. $y$)  in $G^p$ (resp. $G^{p'})$,
%,  
 and setting $\tilde{x}(\zeta)=x(\overline{\zeta})^*$, for $\zeta\in \C^d$,  we observe that  $y(\zeta)x(\overline{\zeta})^*$ is in $\Sl^1$.
 %and the duality between Fock spaces (Theorem ),
 Since $$\left\langle b(\zeta)x(\overline{\zeta}),y(\zeta)\right\rangle_{\text{tr}}=\left\langle P_{\alpha}c(\zeta), y(\zeta)\tilde{x}(\zeta)\right\rangle_{\text{tr}},$$ we have
$$ \left\langle h_{T(b)} x,y\right\rangle_{G^{p}-G^{p'}}=\left\langle P_{\alpha}c, y \tilde{x}\right\rangle_{G^{p}-G^{p'}}
	=\left\langle c,  P_{\alpha}\left(y \tilde{x}\right)\right\rangle_{G^{p}-G^{p'}}
	=\left\langle c,  y \tilde{x}\right\rangle_{G^{p}-G^{p'}}. $$
%\begin{align*}
	%\left\langle h_{T(b)} x,y\right\rangle_{F^{p}_{m,\alpha}(\mathcal{S}^p)-F^{p'}_{m,\alpha}(\mathcal{S}^{p'})}
	%&=\int_{\C^d}\left\langle b(\zeta)x(\overline{\zeta}),y(\zeta)\right\rangle_{\text{tr}}d\mu_{m, \alpha}(\zeta)\\
	%&=\int_{\C^d}\text{tr}\left(b(\zeta)x(\overline{\zeta})y(\zeta)^*\right)d\mu_{m, \alpha}(\zeta)\\
	%&=\int_{\C^d}\text{tr}\left(b(\zeta)\left(y(\zeta)x(\overline{\zeta})^*\right)^*\right)d\mu_{m, \alpha}(\zeta)\\
	%&=\int_{\C^d}\left\langle b(\zeta), y(\zeta)x(\overline{\zeta})^*\right\rangle_{\text{tr}}d\mu_{m, \alpha}(\zeta)\\
	%&=\int_{\C^d}\left\langle P_{\alpha}c(\zeta), y(\zeta)x(\overline{\zeta})^*\right\rangle_{\text{tr}}d\mu_{m, \alpha}(\zeta)\\
	%&=\left\langle P_{\alpha}c, y \tilde{x}\right\rangle_{\alpha}\\
	%&=\left\langle c,  P_{\alpha}\left(y \tilde{x}\right)\right\rangle_{\alpha}\\
	%&=\left\langle c,  y \tilde{x}\right\rangle_{\alpha}.\\
	%	\end{align*}
Using 
$$ \left|\text{tr}\left(c(\zeta)\left(y(\zeta)\tilde{x}(\zeta)  \right)^*\right)\right|\leq \left\|c\right\|_{L^{\infty}\left(\BB(\HH) \right)}\left\|\tilde{x}(\zeta)\right\|_{\Sl^{p}}\left\|y(\zeta)\right\|_{\Sl^{p'}},$$ altogether with H\"{o}lder's inequality and the choice of $c$, we derive that 
\begin{align*}
 \left|\left\langle h_{T(b)} x,y\right\rangle_{G^{p}-G^{p'}}\right|&\leq\left\|c\right\|_{L^{\infty}\left(\BB(\HH) \right)}\left\| \tilde{x} \right\|_{G^p}\left\|y\right\|_{G^{p'}}\\
&\leq 2^{\frac{d}{m}}\left\|b\right\|_{F^{\infty}_{m,\frac{\alpha}{2}}(\be(\HH) )}\left\| \tilde{x} \right\|_{G^{p}}\left\|y\right\|_{G^{p'}},
\end{align*}
%\begin{align*}
%\left|\left\langle h_{T(b)} x,y\right\rangle_{\alpha}\right|&\leq\int_{\C^d}\left|\text{tr}\left(c(\zeta)\left(y(\zeta)\tilde{x}(\zeta)  \right)^*\right)\right|d\mu_{m, \alpha}(\zeta)\\
%&\leq\int_{\C^d}\left\|c(\zeta)\right\|_{\BB(\HH)}\left\|y(\zeta)\tilde{x}(\zeta)\right\|_{\Sl^{1}}d\mu_{m, \alpha}(\zeta)\\
%&\leq\left\|c\right\|_{L^{\infty}\left(\C^d,\BB(\HH) \right)}\int_{\C^d} \left\|y(\zeta)\right\|_{\Sl^{p'}}\left\|\tilde{x}(\zeta)\right\|_{\Sl^{p}} d\mu_{m, \alpha}(\zeta)\\
%&\leq\left\|c\right\|_{L^{\infty}\left(\C^d,\BB(\HH) \right)}\left\|y\right\|_{F^{p'}_{m,\alpha}(\C^d,\Sl^{p'})}\left\| \tilde{x} \right\|_{F^{p}_{m,\alpha}(\C^d,\Sl^{p})}
	%\end{align*}	
	which completes the proof.
\end{proof}

 Theorem \ref{thm A} is now a consequence of  Lemmas \ref{hb bd implies Rdb p}, \ref{de Rdb à b}, \ref{Htb entraine hb}, \ref{hTb bd then hb bd}.

%\section{An application to duality}
%In this section our setting is the classical Segal- Bargmann space, corresponding to he case $m=1$. A weak factThe scalar valued Fock spaces  
%$h_{T(b)}$ is bounded on   $F^{2}_{1,\alpha}(\C^d,\mathcal{S}^2(\mathcal{H}))$ if and only if $h_b$ is bounded on $F^{2}_{1,\alpha}(\C^d,\mathcal{H})$
%and 
%$$ \left\|h_b\right\| = \left\|h_{T(b)}\right\|.  $$
\section{Compactness}\label{compact}
%Let $\HH$ be a separable Hilbert space.
In this section, we assume that $1<p<\infty$. As in section \ref{boundedness}, we study  the compactness of $h_b$ via the compactness of $h_{T(b)}$  on $G^p=F^{p}_{m,\alpha}(\mathcal{S}^q(\mathcal{H}))$, for $q\in\left\{p,2\right\}$. 
% and  $G^p$ will denote one of the spaces $F^{p}_{m,\alpha}(\C^d,\mathcal{S}^q(\mathcal{H}))$, where $q\in\left\{p,2\right\}$.\\
Recall that $\mathcal{K}(\HH)$ denotes the space of compact operators on $\HH$.

\begin{thm B}\label{thm B}
Suppose  $1< p$, and $b\in F^{\infty}_{m,\alpha}( \be(\HH))$.  The following statements are equivalent:

\begin{enumerate}
	\item [(a)] $b\in F^{\infty,0}_{m,\frac{\alpha}{2}}( \mathcal{K}(\HH))$;
	\item [(b)] $h_b$ is compact  on $F^{p}_{m,\alpha}( \HH)$;
	\item [(c)] $h_{T(b)}$ is compact  on $G^p$.
\end{enumerate}
\end{thm B}

 We have shown the equivalence between the membership of $b$ to  $F^{\infty}_{m,\alpha/2}(\mathcal{B}(\HH))$, the boundedness of $h_b$  on $F^{p}_{m,\alpha}(\HH)$ and the boundedness of $h_{T(b)}$ on $G^p$ , with equivalence of norms
\begin{equation}\label{esti norm hb}
	\left\|h_b\right\|\asymp \left\|h_{T(b)}\right\|_{\mathcal{B}(G^p)}\asymp\left\|b\right\|_{F^{\infty}_{m,\alpha/2}(\mathcal{B}(\HH)) )}.
\end{equation}

.

\begin{lemma}\label{cpt sufficiency}
If $b\in F^{\infty,0}_{m,\alpha/2}(\mathcal{K}(\HH))$, then 
%$h_b$ is  compact on $F^{P}_{m,\alpha}(\C^d,\HH)$ and 
$h_{T(b)}$ is compact on $G^p$. 
\end{lemma}

\begin{proof}
 The proof of    Proposition \ref{density pol in F} shows that we can approximate $b$  in the $F^{\infty}_{m,\alpha/2}(\mathcal{B}(\HH))$-norm by a sequence  of $\mathcal{K}(\HH)$-valued polynomials; namely 
  $$ \lim_{N\rightarrow+\infty} \left\|b-\sum_{\left|\nu_j\right|\leq N}\hat p_{N,\nu}v_{\nu}\right\|_{F^{\infty}_{m,\alpha/2}(\mathcal{B}(\HH))}=0,$$
where, for each multiindex $\nu$, $v_{\nu}(\zeta)=\zeta^{\nu}$ and 
%, for each $\nu\in\N^d$,  $v_{\nu}(z)=z^{\nu}$ and 
$\hat p_{N,\nu}$ is in $\mathcal{K}(\HH)$ by (\ref{convolution}). In view of (\ref{esti norm hb}), it is enough to show that $h_{T(\hat p_{N,\nu}v_{\nu})}$
 is a compact operator for each $\nu\in\N^d$. We can  approximate $ \hat p_{N,\nu}$  in the $ \be(\HH)$-norm by finite rank operators $\hat p_{N,\nu,f}$. Then (\ref{esti norm hb}) shows that $h_{T\left(\hat p_{N,\nu} v_{\nu}\right)}$ 
is approximated by $h_{T\left(\hat p_{N,\nu,f}v_{\nu}\right)}$,
 which are finite rank operators on $G^p$.
 %$F^{p}_{m,\alpha}(\C^d,\Sl^2(\HH))$, 
Hence $h_{T(b)}$ is compact.
 %The  reasoning for $F^{p}_{m,\alpha}(\C^d,\Sl^p(\HH))$ is similar.
\end{proof}

Let us now state  a fact which holds in all reflexive Fock or Bergman spaces.
\begin{lemma}\label{cpt necessity}
If $h_{T(b)}$ is compact on $G^{p}$, then $h_b$ is  compact on $F^{p}_{m,\alpha}(\HH)$.
\end{lemma}

\begin{proof}
In this proof,  $r\in\left\{p,p'\right\}$, and $F^r$ stands for $F^{r}_{m,\alpha}(\HH)$.
Let $\left(e_k\right)_{k}$ be an orthonormal  basis of $\HH$, and $ (x_n)_n$ a sequence in $F^{p}$ which converges weakly to $0$. For $y$ a unit vector in $F^{p'}$, and a fixed integer $k$, the proof of  Lemma  \ref{Htb entraine hb} implies that
\begin{align*}
\left|\left\langle h_b x_n,y\right\rangle_{F^{p}-F^{p'}}\right|
%&=\left|\left\langle h_{T(b)} x_n\otimes e_k,y\otimes e_k\right\rangle_{G^{p}-(G^{p})^*}\right|\\
&\leq\left\|h_{T(b)}\left( x_n\otimes e_k\right)\right\|_{G^{p}}\left\|y\otimes e_k\right\|_{(G^{p})^*}.
\end{align*}
It is straightforward to see  that $\left\|y\otimes e_k\right\|_{(G^{p})^*}\leq \left\|y\right\|_{F^{p'}}$ and that the sequence $(x_n\otimes e_k)_{n}$ converges  weakly to $0$ in $ G^{p}$.
Since $h_{T(b)}$ is compact, we have
$ \lim_{n\rightarrow\infty} h_{T(b)}\left(x_n\otimes e_k\right)=0.$
Therefore  $ \lim_{n\rightarrow\infty}  h_b x_n=0$, because of the duality $(F^{p})^*=F^{p'}$ shown in Proposition \ref{dual Fp 1 leq p}. 
This completes the proof.
\end{proof}

The next proposition  characterizes the functions in $F^{\infty,0}_{m,\alpha/2}(Y)$.  We denote by 
 $\mathcal{C}_{0}( Y)$ 
the space of all  $Y$-valued functions defined on $\C^d$, which 
 tend to $0$ as $\left|z\right|\rightarrow\infty$.

\begin{prop}\label{charact F infty 0}
Let $\beta>0$, and suppose that $b$ is an $Y$-valued entire function. Then the following conditions are equivalent:
\begin{enumerate}
	\item [(a)] $b$ is in $F^{\infty,0}_{m,\beta}(Y)$; 
	\item  [(b)] there exists $c\in \mathcal{C}_{0}( Y)$ such that $b=P_{2\beta}c$;
	 \item [(c)] $b$ satisfies
	\begin{equation*}
	\label{FSob infty 0}
\lim_{\left|\zeta\right|\rightarrow +\infty}\left(1+\left|\zeta\right|^{2m}\right)^{-k}\left\|\RR^kb(\zeta)\right\|_Ye^{-\frac{\beta}{2}\left|\zeta\right|^{2m}}=0
\end{equation*}
for any (some) nonnegative integer $k$.
%$b$ in in  $\FF^{k,\infty,0}_{m,\beta}(\C^d, Y)$ for any (some)  nonnegative integer $k$.
\end{enumerate}
 
\end{prop}
\begin{proof}
%[Proof of  Proposition \ref{charact F infty 0}]
Again, we   set $\alpha=2\beta$. 

\textbf{Equivalence $(a) \Leftrightarrow (b)$}. If $b\in F^{\infty,0}_{m,\alpha/2}(Y)$, we  see from the  proof
of   Proposition \ref{charact F infty}  that $b=P_{\alpha}c$, where $c(\zeta)=2^{\frac{d}{m}}b\left(2^{1/m}\zeta\right)e^{-\alpha\left|\zeta\right|^{2m}}$ is in $\CC_0(Y)$. 

For the converse implication, we assume that $c$ is in $\CC_0(Y)$, and take  $\epsilon>0$.  There exists $R>0$ such that
$$\left\| c( \zeta ) \right\|_Y\leq\epsilon,\ \text{ whenever } \left|\zeta\right|> R.$$ 
 Now set $b=P_{\alpha}c$. For $z$  in $\C^d$, we have
$$b(z)=\int_{\left| \zeta \right|\leq R}c(\zeta)K_{m,\alpha}(z,\zeta)d\mu_{m,\alpha}(\zeta)+\int_{\left| \zeta \right|> R}c(\zeta)K_{m,\alpha}(z,\zeta)d\mu_{m,\alpha}(\zeta).  $$
Then we  use (\ref{asymp deri ML m > 1/2}) and Lemma \ref{int blp beta} to obtain
$$\left\|b(z)\right\|_Y\lesssim  \left\|c\right\|_{L^{\infty}(Y)}R^{2d}\left|Rz\right|^{d(m-1)}e^{\alpha\left|Rz\right|^m}+\epsilon e^{\frac{\alpha}{4}\left|z\right|^{2m}}  \lesssim 2\epsilon e^{\frac{\alpha}{4}\left|z\right|^{2m}},$$
as $\left|z\right|\rightarrow\infty,$ 
which shows that  $(b)\Rightarrow (a)$. Thus $(a)\Leftrightarrow (b)$ is proven.\\

%Arguments similar to those of Lemma
%\ref{de Rdb à b} show that  $(3)\Leftrightarrow(1)$.
%To prove the implication $(3)\Rightarrow (1)$,
\textbf{Implication $(c)\Rightarrow (a)$}. 
  Suppose now that    $(c)$ holds,  for  a positive integer $k$, and take $\epsilon>0$. We have 
\begin{equation}\label{radial epsilon}
	\left\|\mathcal{R}^k b(z) \right\|_{Y}\leq\epsilon \left|z\right|^{2km}e^{\frac{\alpha}{4}\left|z\right|^{2m}},\ \text{ whenever }\left|z\right|>R,
\end{equation}
 for some positive constant $R$, and 
\begin{equation}\label{radial f infty}
	\left\|\mathcal{R}^k b(z) \right\|_{Y}\lesssim\left(1+\left|z\right|\right)^{2km}e^{\frac{\alpha}{4}\left|z\right|^{2m}}\ \text{ for all }  z\in\C^d.
\end{equation}
Fix $z $  such that  $\left|z\right|>2R$ and set $a=\frac{\alpha}{4}\left|z\right|^{2m}$.  From the definition of the radial derivative, there is a  positive constant $\eta$ such  that
%$$ \mathcal{R}^k b(tz)\frac{1}{t}=\sum^{d}_{j=1}z_j \left(\partial_{j}\mathcal{R}^{k-1} b \right)(tz),$$ there exists $\eta>0$ such that if $$,
$$ \left\|\mathcal{R}^k b(tz)\frac{1}{t}\right\|_Y\lesssim |z|\ \text{ whenever }t \left|z\right|<\eta.$$
 We next write
\begin{align*}
\left\|\mathcal{R}^{k-1} b(z)-\mathcal{R}^{k-1} b(0) \right\|_{Y}
%&=\left\|\int^{1}_{0}\mathcal{R}^k b(tz)\frac{dt}{t}\right\|_{Y}\\
&\leq\left(\int^{_{\frac{\eta}{\left|z\right|}}}_{0}+\int^{1/2}_{\frac{\eta}{\left|z\right|}}+\int^{1}_{1/2}\right) \left\|\mathcal{R}^k b(tz)\right\|_{Y}\frac{dt}{t}.
% &\leq C{\eta}  +\int^{1}_{\frac{\eta}{\left|z\right|}}\left\|\mathcal{R}^k b(tz)\right\|_{Y}\frac{dt}{t},
\end{align*}
Relation  (\ref{radial f infty}) induces an  estimate for the second term 
 \begin{align*}
\int^{1/2}_{\frac{\eta}{\left|z\right|} }\left\|\mathcal{R}^d b(tz)\right\|_{Y}\frac{dt}{t}
%&\leq C\int^{1/2}_{\frac{\eta}{\left|z\right|} }\left(1+t\left|z\right|\right)^{2km}t^{-1}e^{at^{2m}}dt\\
&\lesssim\int^{1/2}_{\frac{\eta}{\left|z\right|} }\left(t\left|z\right|\right)^{2km}t^{-1}e^{at^{2m}}dt\lesssim\left|z\right|^{2km}e^{\frac{\alpha}{4^{1+m}}\left|z\right|^{2m}}.
%&\leq C'\left|z\right|^{2km}\int^{1}_{\frac{\eta}{\left|z\right|} }\frac{d}{dt}\left[\frac{1}{a}e^{at^{2m}}\right]dt\\
%&\lesssim C'\left|z\right|^{2km}\frac{1}{\left|z\right|^{2m}}e^{\frac{\alpha}{4}\left|z\right|^{2m}}\\
%&\lesssim C'\left|z\right|^{2(k-1)m}e^{\frac{\alpha}{4}\left|z\right|^{2m}},\\
\end{align*}
Now, we use (\ref{radial epsilon}) to handle the third integral
\begin{align*}
\int^{1}_{1/2}\left\|\mathcal{R}^k b(tz)\right\|_{Y}\frac{dt}{t}&\lesssim \epsilon\int^{1}_{1/2}  \left(t\left|z\right|\right)^{2km}t^{-1}e^{at^{2m}}dt\lesssim\epsilon\left|z\right|^{2(k-1)m}e^{\frac{\alpha}{4}\left|z\right|^{2m}}.
\end{align*}
  As $ \left|z\right|\rightarrow\infty$,  we thus see that
$$\left\|\mathcal{R}^{k-1} b(z) \right\|_{Y}=o(1)\left|z\right|^{2(k-1)m}e^{\frac{\alpha}{4}\left|z\right|^{2m}}. $$
By induction, we get
$$ \left\| b(z) \right\|_{Y}=o(1)e^{\frac{\alpha}{4}\left|z\right|^{2m}}, $$ 
which shows $(c)\Rightarrow(a)$.\\

 \textbf{Implication $(a)\Rightarrow (c)$}.
   Let $b$ be in $ F^{\infty,0}_{m, \beta}(Y)$ and  $\epsilon>0$. For some constant $R_1>0$, we have
$$ \left\|b(\zeta)\right\|_Y\leq  \epsilon e^{\frac{\beta}{2}\left|\zeta\right|^{2m}}\ \text{ if }\left|\zeta\right|> R_1 .$$

From the proof of Proposition \ref{charact F infty 0},
%in $F^{\infty}_{m, \beta}(Y)$,
%$$ b(z)=\int_{\C^d}b(\zeta) K_{m, \beta}(z,\zeta)d\mu_{m, \beta}(\zeta),\ z\in\C^d ,$$
 $\RR^k b(z)$ is a linear combination of the integrals
% \ref{R1 b}
\begin{align*}
\RR_l b(z):=\int_{\C^d} b(\zeta)
\left\langle z,\zeta \right\rangle^{l}
 E^{(d-1+l)}_{\frac{1}{m},\frac{1}{m}}
\left(\beta^{\frac{1}{m}}\left\langle z,\zeta\right\rangle\right)e^{-\beta\left|\zeta\right|^{2m}}dv(\zeta),\ 0\leq l \leq k.
%&=C_m\int_{\C^d}b(\zeta)\RR_z E^{(d-1)}_{\frac{1}{m},\frac{1}{m}}\left(\beta^{\frac{1}{m}}\left\langle z,\zeta\right\rangle\right)d\mu_{m, \beta}(\zeta)\\
\end{align*}
Now, Lemma \ref{int blp beta} implies that
\begin{align*}
\left\|\RR_l b(z)\right\|_{Y}
%&\leq C\int_{\C^d}\left\|b(\zeta)\right\|_Y\left|\left\langle z,\zeta\right\rangle\right|
%\left|E^{(d+1-1)}_{\frac{1}{m},\frac{1}{m}}
%\left(\beta^{\frac{1}{m}}\left\langle z,\zeta\right\rangle\right)\right|e^{-\beta\left|\zeta\right|^{2m}}dv(\zeta)\\
&\lesssim\left(\int_{\left|\zeta\right|\leq R_1}+\int_{\left|\zeta\right|> R_1}\right)\left\|b(\zeta)\right\|_Y\left|\left\langle z,\zeta\right\rangle\right|^l
\left|E^{(d-1+l)}_{\frac{1}{m},\frac{1}{m}}
\left(\beta^{\frac{1}{m}}\left\langle z,\zeta\right\rangle\right)\right|e^{-\beta\left|\zeta\right|^{2m}}dv(\zeta)\\
%&\leq C\left( \int_{\left|\zeta\right|\leq R_1}\left\|b\right\|_{F^{\infty}_{m, \beta}(\C^d,Y)} 
%\left|zR_1\right| 
%\left|E^{(d+1-1)}_{\frac{1}{m},\frac{1}{m}}
%\left(\beta^{\frac{1}{m}}\left| z\right| R_1\right)\right| 
%e^{\frac{-\beta}{2}\left|\zeta\right|^{2m}}  dv(\zeta)
% + \epsilon  \int_{\left|\zeta\right|> R_1}\left|\left\langle z,\zeta\right\rangle\right|
%\left|E^{(d+1-1)}_{\frac{1}{m},\frac{1}{m}}
%\left(\beta^{\frac{1}{m}}\left\langle z,\zeta\right\rangle\right)\right|e^{\frac{\beta}{2}\left|\zeta\right|^{2m}} dv(\zeta)\right) \\
%&\leq C'\left(\left\|b\right\|_{F^{\infty}_{m, \beta}(\C^d,Y)} \left|zR_1\right|^{1+(d+1)(m-1)}e^{\beta\left(\left|z\right|R_1\right)^m} \int_{\left|\zeta\right|\leq R_1}|e^{-\beta\left|\zeta\right|^{2m}}  dv(\zeta)+\epsilon I_{m,1,d+1,1,\frac{\beta}{2},\beta}(z)\right)\\
&\lesssim \left(\left\|b\right\|_{F^{\infty}_{m, \beta}(Y)} \left|zR_1\right|^{l+(d+l)(m-1)}e^{\beta\left(\left|z\right|R_1\right)^m} +\epsilon \left|z\right|^{2lm}e^{\frac{\beta}{2}\left|z\right|^{2m}}\right).
%&\lesssim m \beta^{\frac{d+1}{m}} \left\|b\right\| _{F^{\infty}_{m, \beta}(\C^d,Y)} \int_{\C^d}\left|\left\langle z,\zeta\right\rangle\right|\left|E^{(d+1-1)}_{\frac{1}{m},\frac{1}{m}}
%\left(\beta^{\frac{1}{m}}\left\langle z,\zeta\right\rangle\right)\right|e^{-\frac{\beta}{2}\left|\zeta\right|^{2m}}dv(\zeta)\\
%&\lesssim m \beta^{\frac{d+1}{m}} \left\|b\right\| _{F^{\infty}_{m, \beta}(\C^d,Y)} I_{m,1,d+1,1,\frac{\beta}{2},\beta}(z)\\
%&\lesssim m \beta^{\frac{d+1}{m}} \left\|b\right\| _{F^{\infty}_{m, \beta}(\C^d,Y)}\beta^{(d+1)\left(1-\frac{1}{m}\right)-d}\left|z\right|^{2(d+1-d)(m-1)+2}e^{\frac{\beta}{2}\left|z\right|^{2m}}\\
%&\lesssim m \beta \left\|b\right\| _{F^{\infty}_{m, \beta}(\C^d,Y)}\left|z\right|^{2m}e^{\frac{\beta}{2}\left|z\right|^{2m}},
\end{align*}
Then, if $\left|z\right|$ is large enough, 
%$$ \left\|b\right\|_{F^{\infty}_{m, \beta}(\C^d,Y)} \left|zR_1\right|^{1+(d+1)(m-1)}e^{\beta\left(\left|z\right|R_1\right)^m} \int_{\left|\zeta\right|\leq R_1}dv(\zeta)<\epsilon \left|z\right|^{2m}e^{\frac{\beta}{2}\left|z\right|^{2m}},  $$
\begin{align*}
\left\|\RR^k b(z)\right\|_{Y}
%&\leq C'' \left( \epsilon \left|z\right|^{2m}e^{\frac{\beta}{2}\left|z\right|^{2m}}+\epsilon \left|z\right|^{2(d+1-d)(m-1)+2}e^{\frac{\beta}{2}\left|z\right|^{2m}}\right)\\
&\lesssim\epsilon \left|z\right|^{2mk}e^{\frac{\beta}{2}\left|z\right|^{2m}},
\end{align*}
 and we obtain (c).
\end{proof}

\begin{lemma}\label{cpt necessity2}
If $h_b$ is  compact on $F^{p}_{m,\alpha}(\HH)$, then $b\in F^{\infty,0}_{m,\alpha/2}(\mathcal{K}(\HH))$. 
\end{lemma}

\begin{proof}
If $h_b$ is  compact, then it is bounded, and we have shown that $b\in F^{\infty}_{m,\alpha/2}(\mathcal{B}(\HH))$. Let us prove that the Taylor coefficients $\hat  b_{\nu} $ of $b$ are compact operators on $\HH$. Fix $\nu_0$ in $\N^d$, and consider a sequence $(f_k)_{k\in\N}$ in $\HH$ which converges weakly to $0$ as $k\rightarrow\infty$. For each integer $k$, set 
$$x_k(\zeta)= \zeta^{\nu_0}f_k,\qquad  y_k(\zeta)=\hat  b_{\nu_0}f_k,\qquad\zeta\in\C^d.$$
The sequences $(x_k)$ (resp. $(y_k)$) converge weakly to $0$ in $F^{p}_{m,\alpha}(\C^d,\HH)$ (resp. in $F^{p'}_{m,\alpha}(\C^d,\HH)$).
% Indeed, for $y\in F^{p'}_{m,\alpha}(\C^d,\HH)$, 
%$$ \left\langle x_k,y\right\rangle_{F^{p}(\HH)-F^{p}(\HH)}=\int_{\C^d} \left\langle  f_k,y(\zeta)\right\rangle_{\HH} \zeta^{\nu_0} d\mu_{m,\alpha}(\zeta).$$
%Since  for each $\zeta\in\C^d$,
%$$\lim_{k}\zeta^{\nu_0}\left\langle  f_k, y(\zeta)\right\rangle_{\HH}e^{-\alpha\left|\zeta\right|^{2m}}=0$$
 %and
%$$\left|\left\langle f_k, y(\zeta)  \right\rangle_{\HH}\zeta^{\nu_0}\right| e^{-\alpha\left|\zeta\right|^{2m}}\leq\left\|y(\zeta)\right\|_{\HH} \left\|f_k\right\|_{\HH}\left|\zeta^{\nu_0}\right|e^{-\alpha\left|\zeta\right|^{2m}},$$
%by H\"{o}lder inequality and the dominated convergence theorem, we have that $\left\langle x_k,y\right\rangle_{ F^{p}(\HH)-F^{p}(\HH)  }\rightarrow 0$. Similarly, for $x\in F^{p}_{m,\alpha}(\C^d,\HH)$, 
%$$ \left\langle x,y_k\right\rangle_{ F^{p}(\HH)-F^{p}(\HH)  }= \int_{\C^d} \left\langle x(\zeta), \hat  b_{\nu_0} f_k\right\rangle_{\HH} d\mu_{m,\alpha}(\zeta)\rightarrow 0.$$
The compactness of $h_b$ implies that $(h_b x_k)_k$ converges strongly to 0 in $F^{p}_{m,\alpha}(\HH) $, and thus $\left(\left\langle h_b x_k,y_k \right\rangle\right)_k$ converges to 0 as $k\rightarrow\infty.$ Simple computations show that $\left\langle h_b x_k,y_k \right\rangle_{\alpha}=\left\|\hat  b_{\nu_0} f_k\right\|^{2}_{\HH}s_{\alpha,\nu_0}$. 
%\begin{align*}
%	\left\langle h_b x_k,y_k \right\rangle_{F^{p}(\HH)-F^{p}(\HH) }&=\int_{\C^d} \left\langle b(\zeta) \overline \zeta^{\nu_0} f_k, \hat  b_{\nu_0} f_k\right\rangle_{\HH} d\mu_{m,\alpha}(\zeta)\\
%	&=\int_{\C^d} \left\langle \hat  b_{\nu_0}  \zeta^{\nu_0}\overline \zeta^{\nu_0} f_k, \hat  b_{\nu_0} f_k\right\rangle_{\HH} d\mu_{m,\alpha}(\zeta)\\
%	&=\int_{\C^d}\left\|\hat  b_{\nu_0} f_k\right\|^{2}_{\HH}\left|\zeta^{\nu_0}\right|^{2}d\mu_{m,\alpha}(\zeta)\\
%	&=\left\|\hat  b_{\nu_0} f_k\right\|^{2}_{\HH}s_{\alpha,\nu_0},
%\end{align*} 
Therefore, $ \hat b_{\nu_0} f_k\rightarrow 0$ in $\HH$, which proves that $\hat b_{\nu_0}$ is compact on $\HH$, for all multiindex $\nu_0$. Thus, $b$ is a  $\mathcal{K}(\HH)$-valued entire function. It remains to prove that 
$$ \lim_{\left|z\right|\rightarrow\infty} \left\|b(z)\right\|_{\be(\HH)}e^{-\frac{\alpha}{4}\left|z\right|^{2m}}=0.$$
For  $x_0,y_0\in \HH$, $z\in \C^d,$ we shall use the test functions defined in (\ref{test fct}),
%$x_z,y_{z}$ are defined in (\ref{test fct}).
\begin{gather*}
x_z(\zeta)=2^{-\frac{d}{m}}C_m E\left(\alpha^{1/m}\left\langle \zeta,\overline{z}\right\rangle\right)x_0,\\
%y_{l,z}(\zeta)=G_l\left(\alpha^{1/m}\left\langle \zeta,z\right\rangle\right)y_0,\  0\leq l\leq d,\text{ and }
y_z(\zeta)= \sum^{d}_{l=0}a_{l} G_l\left(\alpha^{1/m}\left\langle \zeta,z\right\rangle\right)y_0,\ \zeta\in \C^d.
\end{gather*}
As $|z|\rightarrow\infty$, they satisfy the estimates 
%(see Lemma \ref{int blp beta})
%(see (\ref{esti xz}) and (\ref{esti yz}))
\begin{align*}
%\label{esti test}
\left\|x_z\right\|_{F^{p}_{m,\alpha}(\HH)}&\asymp  \left\|x_0\right\| \left|z\right|^{2(1-\frac{d}{p})(m-1)}e^{\frac{\alpha}{2}\left|z\right|^{2m}},\\
\left\|y_{z}\right\|_{F^{p'}_{m,\alpha}(\HH)}&\asymp  \left\|y_0\right\|\left|z\right|^{2\frac{d}{p}\left(m-1\right)+2}e^{\frac{\alpha}{2}\left|z\right|^{2m}}.
%\text{ as }|z|\rightarrow\infty.
\end{align*}
For some constant  $1/2<c<1$, we have
\begin{multline}\label{Rd b cpct}
\left\langle \mathcal{R} b(2^{1/m}z)x_0,y_0\right\rangle_{\HH}=\left\langle h_b x_z,y_{z}  \right\rangle_{\alpha}
+O(1) \left\|x_0\right\|\left|y_0\right\| \left\|b\right\|_{F^{\infty}_{m,\alpha}\left( \be(\HH)\right)}e^{{c\alpha}\left|z\right|^{2m}}.
\end{multline}

%$x_z,y_{z}$ are defined in (\ref{test fct}).

The weak convergence is denoted by $\rightharpoonup$. We now show that the unit vectors
\begin{align*}
 \tilde x_z&:=\frac{x_z}{\left\|x_z\right\|_{ F^{p}(\HH)  }}\rightharpoonup 0\qquad \text{in } F^{p}_{m,\alpha}(\HH)\\
\text{ and }\tilde y_{z}&:=\frac{y_{z}}{\left\|y_{z} \right\|_{F^{p'}(\HH)} }\rightharpoonup 0\qquad \text{in } F^{p'}_{m,\alpha}(\HH)\text{ as }\left|z\right|\rightarrow\infty.
\end{align*}

  The functions $e_{w,a} (\zeta)=K_{m,\alpha}(\zeta,w)a,\ \text{ for } w\in\C^d,\  a\in \HH, $
%$$e_{w,a} (\zeta)=K_{m,\alpha}(\zeta,w)a,\ \text{ for} w\in\C^d,\  a\in \HH, $$
 induce  bounded linear functionals on $F^{p}_{m,\alpha}(\HH)$ (resp. $F^{p'}_{m,\alpha}(\HH)$), and span a dense subspace in  $F^{p}_{m,\alpha}(\HH)$ (resp. $F^{p'}_{m,\alpha}(\HH)$) by   Proposition \ref{density kernel}. Thus, it is sufficient to prove that
\begin{gather*}
 \left\langle \tilde x_z, e_{w,a}\right\rangle_{\alpha}=\left\langle\tilde x_z(w),a \right\rangle_{\HH}\text{ and }
\left\langle \tilde y_{z}, e_{w,a}\right\rangle_{\alpha}=\left\langle\tilde y_{z}(w),a \right\rangle_{\HH}\text{ , } 
\end{gather*}
tend to $0$ when $\left|z\right|\rightarrow\infty$, and this is true by relation (\ref{asymp deri ML m > 1/2}) and the estimates of $\left\|x_z\right\|_{F^{p}_{m,\alpha}(\HH)}$ and $\left\|y_{z}\right\|_{F^{p'}_{m,\alpha}(\HH)}$. \\

As  $\left|z\right|\rightarrow\infty$,  the compactness of $h_b $ ensures that
$ \left\langle h_b \tilde x_z, \tilde y_{z}\right\rangle_{\alpha}\rightarrow 0,$
 or equivalently
\begin{gather*}
\left\langle h_b  x_z, y_{z}\right\rangle_{\alpha}=o(1)  \left\|x_z\right\|_{F^{p}_{m,\alpha}(\HH)} \left\|y_{z}\right\|_{F^{p'}_{m,\alpha}(\HH)}.
%\left\langle h_b  \xi_z, \eta_{z}\right\rangle_{F^{p}(\HH)-F^{p'}(\HH)}=o(1)  \left\|\xi_z\right\|_{F^{p}(\HH)} \left\|\eta_{l,z}\right\|_{F^{p'}(\HH)},
\end{gather*}
By (\ref{Rd b cpct}), (\ref{esti xz}) and (\ref{esti yz}), we get
$$\left\langle \mathcal{R} b(2^{1/m}z)x_0,y_0\right\rangle_{\HH}= o(1)  \left\|x_0\right\| \left\|y_0\right\|\left|z\right|^{2m}e^{\alpha\left|z\right|^{2m}},$$
and therefore

\begin{equation}\label{Rd b=o()}
	\left\|\mathcal{R} b(z) \right\|_{\be(\HH)}= o(1)  \left|z\right|^{2m}e^{\frac{\alpha}{4}\left|z\right|^{2m}}.
\end{equation}
We conclude by using Proposition \ref{charact F infty 0}.
\end{proof}

Theorem B follows from Lemmas \ref{cpt sufficiency}, \ref{cpt necessity} and \ref{cpt necessity2}.

\section{Further remarks}

We have characterized the boundedness of Hankel operators on $F^{p}_{m,\alpha}(\C^d,\HH)$, $1\leq p<\infty$, as well as their compactness when $1<p<\infty$.
Because of the lack of information on the dual of $F^{\infty}_{m,\alpha}(\C^d,\HH)$, our methods do not apply to  study  $h_b$ on $F^{p}_{m,\alpha}(\C^d,\HH)$  when $p\in\left\{1,\infty\right\}$. It would be interesting to study these cases.\\
When $X$ is a Banach space, we have shown that a necessary condition for $h_b$ to be bounded on $F^{p}_{m,\alpha}(\C^d,X)$ is that $b$ is in $F^{\infty}_{m,\alpha/2}(\C^d,\be(X))$.
The question of knowing whether the converse is true in the non Hilbert case remains open. \\

{\it Acknowledgements.} We would like to thank N. Nikol'skii for suggesting to us the method of proof of Proposition \ref{density pol in F}, and E.-H. Youssfi 
for useful discussions.


\begin{thebibliography}{00}

\bibitem{AlCo}  A. Aleman and O. Constantin, Hankel operators on Bergman spaces and similarity to contractions, {\it Int. Math. Res. Not.} 
{\bf 35} (2004), 1785-1801. 

\bibitem{AlPe} A. Aleman and K.M. Perfekt, Hankel Forms and embedding theorems in
weighted dirichlet spaces, {\it Int. Math. Res. Not.}, {\bf 19} (2012), 4435-4448. 


\bibitem{Ar} N. Aronszajn, Theory of reproducing kernels, {\it Trans. Amer. Math. Soc.} {\bf 68} (1950), 337-404.

\bibitem{AB} J.L. Arregui and O. Blasco Bergman and Bloch spaces of vector-valued functions,
{\it  Math. Nachr.} {\bf 261} (1) (2003), 3-22.

\bibitem{BaEr} H. Bateman and A. Erdelyi, {\it Higher Transcendental Functions}, Vol. 1. McGraw-Hill, New York, 1953.

\bibitem{BL}  J. Bergh and  J. L\"ofstr\"om, {\it Interpolation spaces: an introduction},  Springer-Verlag, Berlin-New York, 1976.
\bibitem{Bo} H. Bommier-Hato, {\it Lipschitz estimates for the Berezin transform}, J. Funct. Spaces Appl. {\bf 8} (2010), 103-128.


\bibitem{BEY} H. Bommier-Hato, M. Engli\v s and E.H. Youssfi, {Bergman-type projections in generalized Fock spaces},  {\it J. Math. Anal. Appl.}, {\bf 329}  (2012), 1086-1104. 


\bibitem{BY} H. Bommier-Hato and E.H. Youssfi, {Hankel operators on weighted Fock spaces}, {\it Integral Equations Operator Theory} {\bf 59} (2007), 1-17.

%\bibitem{Cas} C. Cascante, J. Fabrega, D. Pascuas and J. A. Pel\'aez, Small Hankel operators on generalized Fock spaces,  	arXiv: 1712.05250.



\bibitem{Co} O. Constantin, {Weak product decompositions and Hankel operators on vector-valued Bergman spaces}, {\it J. Operator Theory}
{\bf 59} (2008), 157-178.

\bibitem{CP} O. Constantin and  J. A. Pel\'aez, Integral operators, embedding theorems and a Littlewood-Paley formula on weighted Fock spaces,
{\it J. Geom. Anal.} {\bf 26} (2)(2016), 1109-1154.

\bibitem{Din} N. Dinculeanu, {\it Vector Measures}, 
 International Series of Monographs on Pure and Applied Mathematics {\bf 95},
 1967.
\bibitem{DZ}  M. Dostanic and K. Zhu, Integral operators induced by the Fock kernel,  {\it Integral Equations and Operator Theory} {\bf 60} (2008), 217-236.

\bibitem{FeEv} M.V. Fedoryuk, {\it Asymptotic methods in analysis}, in: M.A. Evgrafov, M.V. Fedoryuk (Eds.), Analysis I. Integral Representations and Asymptotic Methods,
in: R.V. Gamkrelidze (Ed.), Encyclopaedia Math. Sci., Springer, Berlin, Heidelberg, New York, 1989, pp. 83-191.

\bibitem{Foll}  G.B. Folland, {\it  Harmonic Analysis in Phase Space}, Ann. of Math. Stud., vol. 122, Princeton University Press, Princeton, 1989.

\bibitem{Gra} L. Grafakos, {\it  Classical Fourier analysis}, Graduate Texts in Mathematics, 249, Springer-Verlag, New York, 2014.

\bibitem{GoK} I. Gohberg and M.G. Krein, {\it  Introduction to the theory of linear nonselfadjoint
operators}, Amer. Math. Soc., Providence, 1969.

\bibitem{JPR} S. Janson, J. Peetre and R. Rochberg, Hankel forms and the Fock space, {\it  Rev. Mat. Iberoam.} {\bf 3} (1987), 61–138.


\bibitem{Nik} N. Nikol'skii,  {\it  Operators, functions and systems: an easy reading, vol 1: Hardy, Hankel and Toeplitz}, American Mathematical Society, Providence, 2002.
\bibitem{PaKa} R. B. Paris and D. Kaminski, \textit{Asymptotics and Mellin-Barnes Integrals}, Cambridge University Press, Cambridge, 2001.

%\bibitem{Pell1} V.V. Peller, \textit{Vectorial Hankel operators, commutators and related operators of the Schatten-Von-Neumann class $\Sl_p$}

\bibitem{Pell2} V.V. Peller, {\it Hankel Operators and Their Applications}, Springer Monographs in Mathematics, Springer-Verlag, New York, 2003.
\bibitem{Sim} B. Simon, Trace ideals and their applications, Amer. Math. Soc., Providence,
2005.

\bibitem{SY} K. Seip and   E. H. Youssfi, Hankel operators on Fock spaces and related Bergman kernel
estimates, {\it J. Geom. Anal.} {\bf 23} (2013), 170-201.



%\bibitem{WCZ}
\bibitem{WoZh} R. Wong and Y.-Q. Zhao, Exponential asymtotics of the Mittag Leffler function, {\it Constr. Approx.} {\bf 18}  (2002), 355-385.
\bibitem{Zhankel} K. Zhu,  Duality and Hankel operators on the Bergman spaces of bounded symmetric domains, {\it J. Funct. Anal. } {\bf 81} (1988), 260-278.
 \bibitem{Zdisc} K. Zhu,  {\it Operator Theory in Function spaces} Second Edition,  Amer. Math. Soc., 2007.
\bibitem{Zcho}  K. Zhu and H.R. Cho,  Fock-Sobolev spaces and their Carleson measures, {\it J. Funct. Anal.} {\bf 263} (2012), 2483-2506.

\bibitem{Zf} K. Zhu,  {\it Analysis on Fock Spaces}, Springer-Verlag, New York, 2012.
\end{thebibliography}
\end{document}